\documentclass{amsart}[11pt]

\def\volume{\operatorname{vol}}

\def\op{\operatorname}
\def\svolball#1#2{{\volume(\underline B_{#2}^{#1})}}
\def\svolann#1#2{{\volume(\underline A_{#2}^{#1})}}
\def\svolsp#1#2{{\volume(\partial \underline B_{#2}^{#1})}}

\usepackage[usenames,dvipsnames]{color}
\usepackage{amsmath,amssymb,bm,amsthm, amsfonts, mathrsfs, eucal, epsfig}
\usepackage{graphicx}
\usepackage{url}
\usepackage{fancyhdr}
\usepackage[top=2.5cm,bottom=2.0cm,left=2.5cm,right=2.5cm]{geometry}
\allowdisplaybreaks
\makeatletter % '@' now normal "letter"
\@addtoreset{equation}{section}
\makeatother  % '@' is restored as "non-letter"

\begin{document}

% Define new math environments
\newtheorem{Thm}{Theorem}[section]
\newtheorem{Def}[Thm]{Definition}
\newtheorem{Lem}[Thm]{Lemma}
\newtheorem{Rem}[Thm]{Remark}

\newtheorem{Cor}[Thm]{Corollary}
\newtheorem{Prop}[Thm]{Proposition}
\newtheorem{Example}[Thm]{Example}
% Define some short expressions
\newcommand{\g}[0]{\textmd{g}}
\newcommand{\pr}[0]{\partial_r}
\newcommand{\dif}{\mathrm{d}}
\newcommand{\bg}{\bar{\gamma}}
\newcommand{\md}{\rm{md}}
\newcommand{\cn}{\rm{cn}}
\newcommand{\sn}{\rm{sn}}
\newcommand{\seg}{\mathrm{seg}}

\newcommand{\Ric}{\mbox{Ric}}
\newcommand{\Iso}{\mbox{Iso}}
\newcommand{\ra}{\rightarrow}
\newcommand{\Hess}{\mathrm{Hess}}
\newcommand{\RCD}{\mathsf{RCD}}

\title{Almost volume cone implies almost metric cone for annuluses centered at a compact set in $\RCD(K, N)$-spaces}
\author{Lina Chen}
\address[Lina Chen]{Department of mathematics, Nanjing University, Nanjing China}

\email{chenlina\_mail@163.com}
\thanks{Supported partially by NSFC Grant 12001268 and a research fund from Nanjing University.} 
%\date{\today}

\maketitle

\pagestyle{fancy}\lhead{Almost volume cone implies almost metric cone}\rhead{Lina Chen}

\begin{abstract}

\setlength{\parindent}{10pt} \setlength{\parskip}{1.5ex plus 0.5ex
minus 0.2ex} %\noindent

In \cite{CC1}, Cheeger-Colding considered manifolds with lower Ricci curvature bound and gave some almost rigidity results about warped products including almost metric cone rigidity and quantitative splitting theorem. As a generalization of manifolds with lower Ricci curvature bound, for metric measure spaces in $\RCD(K, N)$, $1<N<\infty$, splitting theorem \cite{Gi13} and ``volume cone implies metric cone" rigidity for balls and annuluses of a point \cite{PG} have been proved. In this paper we will generalize Cheeger-Colding's \cite{CC1} result about ``almost volume cone implies almost metric cone for annuluses of a compact subset " to $\RCD(K, N)$-spaces. More precisely, consider a $\RCD(K, N)$-space $(X, d, \mathfrak m)$ and a Borel subset $\Omega\subset X$.  If the closed subset $S=\partial \Omega$ has finite outer curvature, the diameter $\op{diam}(S)\leq D$ and the mean curvature of $S$ satisfies
$$m(x)\leq m, \, \forall x\in S,$$
and
$$\mathfrak m(A_{a, b}(S))\geq  (1-\epsilon)\int_a^b \left(\op{sn}'_H(r)+ \frac{m}{n-1}\op{sn}_H(r)\right)^{n-1}dr \mathfrak m_S(S)$$
then $A_{a', b'}(S)$ is measured Gromov-Hausdorff close to a warped product $(a', b')\times_{\op{sn}'_H(r)+ \frac{m}{n-1}\op{sn}_H(r)}Y$, $A_{a, b}(S)=\{x\in X\setminus \Omega, \, a<d(x, S)<b\}$, $a<a'<b'<b$, $Y$ is a metric space with finite components with each component a $\RCD(0, N-1)$-space when $m=0, K=0$ or  a $\RCD(N-2, N-1)$-space for other cases and $H=\frac{K}{N-1}$. Note that when $m=0, K=0$, our result is a kind of quantitative splitting theorem and in other cases it is an almost metric cone rigidity. 

To prove this result, different from \cite{Gi13, PG}, we will use \cite{GiT}'s second order differentiation formula and a method similar as \cite{CC1}.
 \end{abstract}

\section{Introduction} 

 %Consider Riemannian manifolds with lower Ricci curvature bound, Cheeger-Colding \cite{CC1}  showed that the quantitative splitting theorem which says that for manifolds with almost nonnegative lower Ricci curvature bound, small bfor a $n$-manifold $M$ with Ricci curvature $\op{Ric}_M\geq $ 
In \cite{CC1}, Cheeger-Colding gave the following ``almost volume cone implies almost metric cone"  rigidity.

\begin{Thm}[\cite{CC1}] \label{ann-rig}
%Given $0\leq a<b, \alpha>0$ and $m, H, N$, there is $\epsilon(H, N, m, a, b, \alpha)>0$ such that for $0<\epsilon<\epsilon(H, N, m, a, b, \alpha)$, 

If a complete $N$-manifold $M$ with $\op{Ric}_M\geq (N-1)H$ and a compact subset $\Omega\subset M$ satisfies that the mean curvature
\begin{equation} m(x)\leq m, \forall \, x\in S=\partial \Omega, \label{mean-com}\end{equation}
$$\op{diam}(S)\leq D,$$
and 
\begin{equation} \volume(A_{a, b}(S))\geq (1-\epsilon) \int_a^b \left(\op{sn}'_H(r)+ \frac{m}{N-1}\op{sn}_H(r)\right)^{n-1}dr \volume(S),\label{alm-max}\end{equation}
then 
$$d_{GH}(A_{a+\alpha, b-\alpha}(S), (a+\alpha, b-\alpha)\times_{\op{sn}'_H(r)+ \frac{m}{N-1}\op{sn}_H(r)} Y)\leq \Psi(\epsilon | N, H, m, a, b, \alpha, D).$$
where $A_{a, b}(S)=\{x\in X\setminus \Omega, \, a<d(x, S)<b\}$, $Y$ is a length metric space with at most $C(N, H, a, b, D)$ components $Y_i$ such that $\op{diam}(Y_i)\leq c(N, H, m, a, b, \alpha, D)$ and 
$$\op{sn}_H(r)=\begin{cases} \frac{\sin\sqrt{H}r}{\sqrt H}, & H>0;\\ r, & H=0;\\ \frac{\sinh \sqrt{-H}r}{\sqrt{-H}}, & H<0. \end{cases}$$
\end{Thm}

If $S$ is a hypersurface of $M$, by Heintze-Karcher \cite{HK}, \eqref{mean-com} implies that 
\begin{equation} \volume(A_{a, b}(S))\leq \int_a^b \left(\op{sn}'_H(r)+ \frac{m}{N-1}\op{sn}_H(r)\right)^{N-1}dr \volume(S).\label{loc-vol}\end{equation}
In particular, if $M$ is compact with $\tilde D=\op{diam}(X)$,
\begin{equation}
\volume(M)\leq \int_{[-\tilde D, \tilde D]} \left(\op{sn}'_H(r)+ \frac{m}{N-1}\op{sn}_H(r)\right)^{N-1}dr \volume(S). \label{gvol-com}\end{equation}
And when $H>0$, the equality holds in \eqref{gvol-com}  iff $M$ and $N$ have constant curvature (see \cite{HK}).

In \cite{Ket2}, Ketterer extended Heintze-Karcher \cite{HK}'s results about the volume comparison \eqref{loc-vol} and the rigidity result for $K>0$ in $\RCD(K, N)$-spaces with $S=\partial \Omega$, $\Omega$ is Borel and $H=\frac{K}{N-1}$. 

In this note, we will generalize Theorem~\ref{ann-rig} to $\RCD(K, N)$-spaces which can also be treated as a quantitative version and a generalization of Ketterer \cite{Ket2}'s work (see Theorem~\ref{HK-general}) to arbitrary $K$.
In the following, we will use the same definitions of mean curvature, finite outer curvature and measure on $S$, $\mathfrak m_S$ as in \cite{Ket2} (see Section 2.7 for these definitions). 
\begin{Thm} \label{main} 
%Given $0\leq a<b$ and $N\geq 3, m, K$, there is $\epsilon(H, N, m, a, b)>0$ such that for $0<\epsilon<\epsilon(K, N, m, a, b, )$, 

If a metric measure space $(X, d, \mathfrak m)\in \op{RCD}(K, N)$, $3\leq N<\infty$, $\op{supp}(\mathfrak m)=X$ and a Borel subset $\Omega\subset X$ satisfies that $S=\partial \Omega$ is closed,  $\op{diam}(S)\leq D$, $\mathfrak m(S)=0$, $S$ has finite outer curvature, the mean curvature
\begin{equation}m(x)\leq m, \, \forall \, x\in S, \label{mean-bound}\end{equation}
and
\begin{equation} \mathfrak m(A_{a, b}(S))\geq  (1-\epsilon)\int_a^b \left(\op{sn}'_H(r)+ \frac{m}{N-1}\op{sn}_H(r)\right)^{n-1}dr \mathfrak m_S(S), \label{lvol-com}\end{equation}
then
$$d_{mGH}(A_{a', b'}(S), (a', b')\times_{\op{sn}'_H(r)+ \frac{m}{N-1}\op{sn}_H(r)} Y)\leq \Psi(\epsilon | N, K, m, a, b, D),$$ 
where $H=\frac{K}{N-1}$, $a'=a+(b-a)/3, b'=b-(b-a)/3$ and $(Y, d_Y, \mathfrak m_Y)$ has at most $C(N, K, a, b, D)$ components $Y_i$ with each $Y_i\in \RCD(0, N-1)$ for $m=0, K=0$, $Y_i \in \RCD(N-2, N-1)$ for $m\neq 0$ or $K\neq 0$.
\end{Thm}

\begin{Rem}
(i) When $K=0, m=0$, \eqref{lvol-com} becomes 
\begin{equation}\mathfrak m(A_{a, b}(S))\geq (1-\epsilon)(b-a)\mathfrak m_S(S)\label{lvol-com-imp0}\end{equation}
and Theorem~\ref{main} is a kind of quantitative splitting theorem. 

In \cite{Hu}, Huang gave a quantitative splitting rigidity under \eqref{lvol-com-imp0} and a measure-decreasing-along-distance-function (MDADF) condition. The assumption that $S$ has finite outer measure and mean curvature upper bound \eqref{mean-bound} implies MDADF condition (see Lemma~\ref{vol-ele-com} and the definition of finite outer curvature). A better result we have is that $Y$ can be chosen as each component in $\RCD(0, N-1)$.

(ii) When $K\neq 0$ or $m\neq 0$, assume 
\begin{equation}m=(N-1)\frac{\op{sn}'_H(r_0)}{\op{sn}_H(r_0)},\label{mean-imp}\end{equation}
then \eqref{lvol-com} becomes
 \begin{equation} \frac{\mathfrak m(A_{a, b}(S))}{\mathfrak m_S(S)}\geq  (1-\epsilon)\frac{\svolann{H}{a+r_0, b+r_0}}{\svolsp{H}{r_0}}. \label{lvol-com-imp1}\end{equation}

In \cite{PG}, Philippis-Gigli pointed out that  ``volume cone implies metric cone" holds for annulus centered at a point where they assume  
$$\frac{\mathfrak m(A_{a, b}(x))}{\mathfrak m(\partial B_a(x))}\geq (1-\epsilon)\frac{\svolann{H}{a, b}}{\svolsp{H}{a}},$$
and derived a quantitative rigidity as in Theorem~\ref{main}.
Here $$\mathfrak m(\partial B_a(x))=\limsup_{\delta\to 0}\frac{\mathfrak m(A_{a, a+\delta}(x))}{\delta}.$$

Our result is a generalization of \cite{PG} in some sense.

(iii) When the equality holds in \eqref{lvol-com}, we have that $A_{a', b'}(S)$ has a warped product structure and $\partial B_{a'}(S)$ has constant mean curvature.
\end{Rem}

Now we give a sketch of the proof of Theorem~\ref{main}. Consider a sequence of $\RCD(K, N)$-spaces, $(X_i, d_i, \mathfrak m_i)$ which is measured Gromov-Hausdorff convergent to a $\RCD(K, N)$-space $(X, d, \mathfrak m)$. Assume $S_i=\partial \Omega_i, \Omega_i\subset X_i$ is Borel with $\mathfrak m_i(S_i)=0$ and $\op{diam}(S_i)\leq D$.
Define a signed distance function associated to $\Omega_i$.
$$d_{s, i}(x)=\begin{cases} d_i(x, S_i), & x\in X_i\setminus \Omega_i;\\ -d_i(x, S_i), & x\in \Omega_i. \end{cases}$$
Obviously, $d_{s, i}$ is $1$-Lipschitz.  
Then by \cite[Proposition 2.70]{Vi} or \cite[Proposition 2.12]{MN}, there is a $1$-Lipschitz function $d_s: X\to \Bbb R$ such that 
$d_{s, i}$ converges uniformly to $d_s$ on any compact set.  Let $S=\{x\in X, \, d_s(x)=0\}$, $\Omega=\{x\in X, \, d_s(x))\leq 0\}$. Then as in \cite[Lemma 3.25]{Hu}, we know that $d_s$ is a signed distance function associated to $\Omega$.

Assume $S_i$ has finite outer curvature and the mean curvature
$$m(x)\leq m, \forall\, x\in S_i,$$
\begin{equation} \mathfrak m(A_{a, b}(S_i))\geq  (1-\epsilon_i)\int_a^b \left(\op{sn}'_H(r)+ \frac{m}{n-1}\op{sn}_H(r)\right)^{n-1}dr \mathfrak m_{S_i}(S_i), \,  \epsilon_i\to 0. \label{lqvol-com}\end{equation}

To prove Theorem~\ref{main}, we only need to show that $A_{a', b'}(S)$ is isometric to $(a', b')\times_{\op{sn}'_H(r)+ \frac{m}{n-1}\op{sn}_H(r)} Y$ where $(Y, d_Y, \mathfrak m_Y)\in \RCD(0, N-1)$ for $m=0, K=0$, $(Y, d_Y, \mathfrak m_Y)\in \RCD(N-2, N-1)$ for the other cases.

To obtain this result, first we have that:

($\ast$) With intrinsic metric $A_{a', b'}(S)$ is isometric to a warped product $(a', b')\times_{\op{sn}'_H(r)+ \frac{m}{n-1}\op{sn}_H(r)} Y$ (for the definition see Section 2.6).  

We will follow the process as in \cite{CC1}. The assumption $S$ has finite outer curvature and $m(X)\leq m$, together with the laplacian formula derived by \cite{CMo},  we will derive the laplacian comparison of $d_s$ and relative volume comparison (see Lemma~\ref{glap-com} and Lemma~\ref{rel-vol}). Then by the volume condition \eqref{lqvol-com}, we will get a laplacian estimate of $d_s$ in Theorem~\ref{lap-main}. These laplacian estimates and improved Bochner's inequality in $\RCD$-spaces (\cite{RS, Han18}, see also Theorem~\ref{Boc-ine})  give the Hessian estimates Theorem~\ref{hess-est}. Then using the second order differentiation in $\RCD(K, N)$-space \cite{GiT}, we can show that the metric in the path connected component of $A_{a', b'}(S)$ satisfies the Pythagoras theorem when $m=0$, $K=0$ and Cosine law for the other cases. This gives the warped product structure of $A_{a', b'}(S)$. And the relative volume comparison gives that $Y$ has at most $C(N, K, D, b, a)$ components. 

Assume $Y$ has one component. Then ($\ast$) and that $A_{a', b'}(S)\subset (X, d, \mathfrak m)\in \RCD(K, N)$ enable us to derive that:

($\ast\ast$) $(Y, d_Y, \mathfrak m_Y)\in \RCD(0, N-1)$ for $m=0, K=0$ and $(Y, d_Y, \mathfrak m_Y)\in \RCD (N-2, N-1)$ for the other cases (see Section 6).

Endow $Y$ with an admissible metric and an admissible measure from the warped product structure ($\ast$). A similar argument as in \cite{CDNPSW} shows that $(Y, d_Y, \mathfrak m_Y)$ is infinitesimally Hilbertian and satisfies Sobolev to Lipschitz property (see Theorem~\ref{measure-pro}).  Now by local to global property we can see that $\Bbb R\times Y$ (when $m=0, K=0$) and the Euclidean cone $C(Y)$ (when $m\neq 0, K=0$) are $\RCD(0, N)$-spaces.  Then \cite{Gi13} and \cite{Ket2} gives ($\ast\ast$) for $K=0$. For $K\geq 0$, we will follow the argument in the proof of \cite[Theorem 1.2]{Ket2} where Ketterer showed that if the $(K, N)$-cone $(C(Y), d_K, \mathfrak m_N)$ is a $\RCD(K, N)$-space, then $Y$ is $\RCD(N-2, N-1)$-space.

The paper is organized as follows. In Section 2, we will present some basic definitions and facts we need in the poof of Theorem~\ref{main}. Then by studying the relative volume comparison for annuluses centered at a compact subset, we give the Laplacian estimates of the distance function from the compact subset in Section 3. Then in Section 4, we will give the corresponding Hessian estimates in $A_{a', b'}(S)$. In Section 5, by the second differential formula in $\RCD(K, N)$-spaces\cite{GiT}, using the Hessian estimates and a methods as in \cite{CC1} we will derive that with the intrinsic metric the annulus  $A_{a', b'}(S)$ satisfies Pythagoras theorem or Cosine law. In Section 6, we will give the warped product structure of $A_{a', b'}(S)$ and by studying the properties of  the section $Y$ of the warped product, we will prove that $Y$ is a $\RCD$-space.

The author would like to thank Professor Xian-tao Huang's advice about  that $Y$ may contain more than one components in the main results. 

\section{Preliminary}

In this section, we recall some basic definitions and properties that we need in the proof of Theorem~\ref{main}. Let $(X, d, \mathfrak m)$ be a metric measure space satisfying that $(X, d)$ is a complete, separable and locally compact geodesic metric space endowed with a nonnegative Radon measure $\mathfrak m$ which is supported on $X$ and is finite on any bounded sets. We refer readers to the survey \cite{Am} for an overview of the topic and bibliography about curvature-dimension bounds in metric measure spaces.
% and refer \cite{FY} for the topic about equivariant Gromov-Hausdorff convergence. 

\subsection{Calculus in metric measure spaces} For the details of this subsection one can confer \cite{Gi13}.

Consider a metric measure space $(X, d, \mathfrak m)$ as above. %A curve $\gamma: [0,1]\to X$ is called a constant speed geodesic if $d(\gamma(s), \gamma(t))=|s-t|d(\gamma(0), \gamma(1))$, for $s, t\in [0,1]$.Let $\op{Geo}(X)$ be the class of constant speed geodesics in $X$.  
Let $C([0, 1], X)$ be the space of continuous curves with weak convergence topology and let $\mathcal{P}(C([0,1], X))$ be the space of Borel probability measures of $C([0,1], X)$. A measure $\pi\in \mathcal{P}(C([0,1], X))$ is called a \textbf{test plan} if for some $c>0$, 
$$(e_t)_{\sharp}(\pi)\leq c \mathfrak m, \forall \, t\in [0, 1], \quad \int\int_0^1 |\dot\gamma(t)|dt d\pi(\gamma)<\infty,$$
where $|\dot\gamma(t)|=\lim_{h\to 0}d(\gamma(t+h), \gamma(t))/|h|$ and $e_t: C([0, 1], X) \to X$, $e_t(\gamma)=\gamma(t)$ is the evaluation map.
 Sobolev class  $S^2(X, d, \mathfrak m)$ is defined as the space of $f: X\to \Bbb R$, such that there exists $G\in L^2(X, \mathfrak m)$,
 $$\int |f(\gamma(1))-f(\gamma(0))|d\pi(\gamma)\leq \int\int_0^1G(\gamma(t))|\dot\gamma(t)|dt d\pi(\gamma), \,   \forall \, \text{test plan } \pi, $$
where $G$ is called a weak upper gradient of $f$. Let $|\nabla f|_w$ be the minimal (in $\mathfrak m$-a.e. sense) weak upper gradient of $f$.

The space $W^{1,2}(X, d, \mathfrak m)=L^2(X, \mathfrak m)\cap S^2(X, d, \mathfrak m)$ endowed with the norm
$$\|f\|^2_{W^{1,2}}=\|f\|_{L^2}^2+ \||\nabla f|_w\|^2_{L^2},$$
is a Banach space.

Define the \textbf{Cheeger energy} as $\op{Ch}: L^2(X, \mathfrak m)\to [0, \infty]$
$$\op{Ch}(f)=\begin{cases} \frac12\int |\nabla f|_w^2d\mathfrak m, & f\in W^{1,2}(X, d, \mathfrak m)\\
 +\infty, & \text{ otherwise}. \end{cases}$$

We say $(X, d, \mathfrak m)$ is \textbf{infinitesimally Hilbertian} if $W^{1,2}(X, d, \mathfrak m)$ is a Hilbert space, i.e., the Cheeger energy is a quadratic form.

%In this subsection, we always assume that $(X, d, m)$ is an infinitesimally Hilbertian space. 
In the following of this section we always assume that  $(X, d, \mathfrak m)$ is infinitesimally Hilbertian. 

For an open subset $\Omega\subset X$, let $W^{1,2}_{\op{loc}}(\Omega)$ be the space of function $f: \Omega\to \Bbb R$ that locally equal to some function in $W^{1, 2}(X, d, \mathfrak m)$.
For $f, g\in W^{1,2}_{\op{loc}}(\Omega)$, define
$$\Gamma(f, g)=\left<\nabla f, \nabla g\right>=\liminf_{\epsilon\downarrow 0}\frac{|\nabla (g+\epsilon f)|_w^2-|\nabla g|_w^2}{2\epsilon}.$$
In fact $\left<\nabla f, \nabla g\right>$ can be achieved by taking limit directly $\mathfrak m$-a.e(cf. \cite{Gi14}). By \cite{Gi14}, the map $\Gamma: W^{1,2}_{\op{loc}}(\Omega)\times W^{1,2}_{\op{loc}}(\Omega)\to L^1_{\op{loc}}(\Omega)$ is symmetric, bilinear and $\Gamma(f, f)=|\nabla f|_w^2$.

% Now we define the Laplacian. 
 \begin{Def}
 For $f\in W^{1,2}_{\op{loc}}(\Omega)$, if there exists a Radon measure $\mu$ on $\Omega$ such that 
$$-\int\left<\nabla f, \nabla g\right>=\int gd\mu$$
holds for any Lipschitz function $g: \Omega\to \Bbb R$, $\op{supp}g\subset\subset \Omega$, then $\mu$ is called the \textbf{distributional Laplacian} or \textbf{measure valued Laplacian} of $f$ and denote it by $\left.\Delta f\right|_{\Omega}$.
\end{Def}
 Let $D(\Delta, \Omega)$ be the space of $f$ which has a distribution Laplacian. By the property of $\Gamma$, we know that $D(\Delta, \Omega)$ is a vector space and the Laplacian is linear. For $f\in W^{1,2}(X, d, \mathfrak m)\cap  D(\Delta, X)$, if $\Delta f= h\mathfrak m$, $h\in L^2(X, \mathfrak m)$, we denote $\Delta f=h$.
%\begin{Prop}[\cite{Gi}]
%Assume $(X, d, m)$ is infinitesimally Hilbertian. For an open set $\Omega\subset X$, a Lipschitz function $f: \Omega\to \Bbb R$, if there exists a Radon measure $\mu$ on $\Omega$ such that for any Lipschitz function $g: \Omega\to \Bbb R_{\geq 0}$ with $\op{supp}g\subset\subset \Omega$,
%$$-\int\left<\nabla f, \nabla g\right>dm\leq \int gd\mu,$$
%then $f\in D(\Delta, \Omega)$ and $\Delta f\leq \mu$.
%\end{Prop}

\subsection{Tangent and cotangent modules} The details of this subsection can be found in \cite{Gi18}.

Consider a measured space $(X, \mathcal A, \mathfrak m)$ where $\mathcal A$ is its $\sigma$-algebra.  Let $\mathcal B(X)=\mathcal A/\sim$, where $A, B\in \mathcal A, A\sim B$ iff $\mathfrak m((A\setminus B)\cup (B\setminus A))=0$.
A Banach space $(\mathcal M, \|\cdot\|)$ is called a \textbf{$L^{\infty}(X, \mathfrak m)$-premodule} if there is a bilinear map 
$$L^{\infty}(X, \mathfrak m)\times \mathcal M\to \mathcal M, \quad (f, v)\mapsto f\cdot v,$$
such that for each $v\in \mathcal M$, $f, g\in L^{\infty}(X, \mathfrak m)$,
$$(fg)\cdot v=f\cdot (g\cdot v), \, 1\cdot v=v, \, \|f\cdot v\|\leq \|f\|_{L^{\infty}(X,\mathfrak m)}\|v\|.$$
An $L^{\infty}(X,\mathfrak m)$-premodule $(\mathcal M, \|\cdot\|)$ is called a \textbf{$L^{\infty}(X, \mathfrak m)$-module} if 

(1) Locality: for each $x\in \mathcal M$, $A_n\in \mathcal B(X)$, 
$$\forall n, \, \chi_{A_n}\cdot v=0  \Rightarrow \chi_{\cup_n A_n}\cdot v=0;$$

(2) Gluing: for every sequence $\{v_n\}\subset \mathcal M$, $\{A_n\}\subset \mathcal B(X)$, if 
$$\chi_{A_i\cap A_j}\cdot v_i=\chi_{A_i\cap A_j}\cdot v_j, \forall i,j, \quad \limsup_{n\to \infty}\|\sum_{i=1}^n\chi_{A_i}\cdot v_i\|<\infty,$$
then there is $v\in \mathcal M$, 
$$\chi_{A_i}\cdot v=\chi_{A_i}\cdot v_i, \forall i, \quad \|v\|\leq \liminf_{n\to \infty}\|\sum_{i=1}^n\chi_{A_i}\cdot v_i\|.$$

A \textbf{module morphism} is a map $T: \mathcal M_1\to \mathcal M_2$ which is bounded and linear by viewing $\mathcal M_1$ and $\mathcal M_2$ as Banach spaces and satisfies the locality condition
$$T(f\cdot v)=f\cdot T(v), \, \forall v\in \mathcal M_1, f\in L^{\infty}(X, \mathfrak m).$$
Denote all module morphism from $\mathcal M_1$ to $\mathcal M_2$ by $\op{Hom}(\mathcal M_1, \mathcal M_2)$. The \textbf{dual module} $\mathcal M^*=\op{Hom}(\mathcal M,  L^1(X,\mathfrak m))$.

If there is a non-negative map $|\cdot| : \mathcal M\to L^p(X, \mathfrak m)$, $p\in [0, \infty]$ satisfying that 
$$\||v|\|_{L^p(X, \mathfrak m)}=\|v\|, \quad |f\cdot v|=|f||v|, \mathfrak m-a.e., \forall \, v\in \mathcal M, f\in L^{\infty}(X,\mathfrak m),$$ 
then $\mathcal M$ is called a \textbf{$L^p(X,\mathfrak m)$-normed $L^{\infty}(X,\mathfrak m)$-premodule (resp. module)} when $\mathcal M$ is a $L^{\infty}(X, \mathfrak m)$-premodule (resp. module).  $|\cdot|$ is called the pointwise $L^{p}(X,\mathfrak m)$-norm.
And locally, for $A\in \mathcal B(X)$, we can define $\left.\mathcal M\right|_A=\{v\in \mathcal M, \, |v|=0\,\, \mathfrak m-a.e. \text{ on }A^c\}$ which is a $L^p(X,\mathfrak m)$-normed $L^{\infty}(X,\mathfrak m)$-module.
A $L^2(X,\mathfrak m)$-normed $L^{\infty}(X,\mathfrak m)$-module which is a Hilbert space under $\|\cdot\|$ is called a \textbf{Hilbert module}.

Consider a $L^2(X,\mathfrak m)$-normed $L^{\infty}(X,\mathfrak m)$-module $\mathcal M$, $V\subset \mathcal M$. Let $\op{Span}(V)$ be the collection of $v\in \mathcal M$ such that there is a Borel decomposition $\{X_n\}$ of $X$, and for each $n$, there are $v_{1, n}, \cdots, v_{k_n, n}\in V$, $f_{1, n}, \cdots, f_{k_n, n}\in L^{\infty}(X,\mathfrak m)$, 
$$\chi_{X_n}v=\sum_1^{k_n}f_{i, n}v_{i, n}, \forall n.$$
And we say $V$ generates $\mathcal M$ if $\overline{\op{Span}(V)}=\mathcal M$.

\begin{Def}[\cite{Gi18}]
There is a unique, up to isomorphism, Hilbert module $L^2(T^*X)$ endowed with a linear map $d: W^{1, 2}(X, d, \mathfrak m)\to L^2(T^*X)$ satisfying
$$|df|=|\nabla f|_w, \mathfrak m-a.e., \forall f\in W^{1,2}(X, d, \mathfrak m); \quad d(W^{1,2}(X, d, \mathfrak m)) \text{ generates }L^2(T^*X).$$
We call $L^2(T^*X)$ the cotangent module of $(X, d, \mathfrak m)$. The dual of $L^2(T^*X)$ is called the tangent module of $(X, d, \mathfrak m)$ and denoted by $L^2(TX)$. Elements of $L^2(TX)$ is called vector fields. And denote by $\nabla f$ the dual of $df$.
\end{Def}

Let $D(\op{div})\subset L^2(TX)$ be the space of vector fields $v$ satisfying that there is $f\in L^2(X, \mathfrak m)$ such that for any $g\in W^{1,2}(X, d, \mathfrak m)$, 
$$\int fg d\mathfrak m=-\int dg(v) d\mathfrak m.$$
$f$ is called the \textbf{divergence} of $v$ and denoted by $\op{div}(v)$. If $f\in D(\Delta)$, then $\nabla f\in D(\op{div})$ and $\op{div}(\nabla f)=\Delta f$ (see \cite[Proposition 2.3.14]{Gi18}).

For two Hilbert module $\mathcal H_1, \mathcal H_2$, we can define the tensor product $\mathcal H_1\otimes \mathcal H_2$ and the exterior product $\mathcal H_1\wedge \mathcal H_2$ (see Section 1.5 in \cite{Gi18}). And denote the pointwise $L^2(X, \mathfrak m)$-normal of the tensor product $L^2((T^*)^{\otimes 2}X)$ by $|\cdot|_{\op{HS}}$.

\subsection{$\op{CD}(K, N)$-spaces and $\RCD(K, N)$-spaces} %Here we quickly recall some basic definitions and properties of $\RCD^*(K, N)$-spaces. 

In this subsection, we recall the definitions of $\op{CD}(K, N)$-spaces and $\RCD(K, N)$-spaces. The notion of curvature dimension condition $\op{CD}(K, N)$ was introduced by Lott-Villani (\cite{LV}) and Strum (\cite{St1, St2}) independently. The Riemannian curvature dimension condition $\RCD(K, N)$ was introduced by a series of works \cite{AGS, Gi15,  EKS, CMi, AMS19}.

Let $(X, d, \mathfrak m)$ be as in the beginning of this section. Let $\mathcal{P}_2(X)$ be the space of Borel probability measures $\mu$ on $(X, d)$ satisfying $\int_X d(x_0, x)^2d\mu(x)<\infty $ for some $x_0\in X$. For $\mu, \nu\in \mathcal{P}_2(X)$, define
$$W_2(\mu, \nu)=\left(\inf\int\int_0^1|\dot\gamma(t)|^2dt d\pi(\gamma)\right)^{\frac12},$$
where the infimum is taken among all $\pi\in \mathcal{P}(C([0,1], X))$ with $(e_0)_*(\pi)=\mu$, $(e_1)_*(\pi)=\nu$. In fact, the minimal can always be achieved. We call the plan $\pi$ which achieves the minimal an \textbf{optimal transportation} and denote the set of optimal transportations by $\op{OpGeo}(\mu, \nu)$. %In fact, $W_2$ is a distance on $\mathcal{P}_2(X)$, and $(\mathcal{P}_2(X), W_2)$ is a geodesic space provided $(X, d)$ is a geodesic space (cf. \cite{Vi}).

For $N\geq 1, K$, let $\sigma_{K, N}: [0, 1]\times \Bbb R^{+}\to \Bbb R$ be as
$$\sigma_{K, N}^{t}(\theta)=\left\{\begin{array}{cc}
+\infty, & K\theta^2\geq N\pi^2,\\
\frac{\sin(t\theta\sqrt{K/N})}{\sin(\theta\sqrt{K/N})}, & 0<K\theta^2<N\pi^2,\\
t, & K\theta^2=0,\\
\frac{\sinh(t\theta\sqrt{-K/N})}{\sinh(\theta\sqrt{-K/N})},& K\theta^2<0.\end{array}\right.$$
and let 
$$\tau_{K, N}^{t}(\theta)=t^{\frac1N}\sigma_{K, N-1}^t(\theta)^{\frac{N-1}{N}}.$$

\begin{Def}[\cite{LV, St1, St2}]
Given $K\in \Bbb R, N\geq 1$, we say a metric measure space $(X, d, \mathfrak m)$ is a $\op{CD}(K, N)$-space if for any two measures $\mu_0, \mu_1\in \mathcal{P}_2(X)$ with bounded support which contains in $\mathfrak m$'s support, there exists $\pi\in \op{OpGeo}(\mu_0, \mu_1)$ such that for each $t\in [0, 1]$
$$-\int\rho_t^{1-\frac{1}{N}}dm\leq -\int \tau_{K, N}^{1-t}(d(\gamma(0),\gamma(1)))\rho_0^{-\frac{1}{N}}(\gamma(0))+\tau_{K, N}^t(d(\gamma(0),\gamma(1)))\rho_1^{-\frac{1}{N}}(\gamma(1))d\pi(\gamma),$$
where $(e_t)_{\sharp}\pi=\rho_t\mathfrak m + \mu_t, \mu_t\bot \mathfrak m$. We call $(X, d, \mathfrak m)$ is a $\op{CD}^*(K, N)$-space if the above inequality holds for $\sigma^t_{K, N}$ instead of $\tau^t_{K, N}$.
\end{Def}

\begin{Def}[\cite{AGS, Gi15}] 
A metric measure space $(X, d, \mathfrak m)$ is a $\RCD(K, N)$-space (resp. $\RCD^*(K, N)$-space)  if it is an infinitesimally Hilbertian $\op{CD}(K, N)$-space (resp. $\op{CD}^*(K, N)$-space).
\end{Def}

And we say a metric measure space $(X, d, \mathfrak m)$ is a $\op{CD}_{\op{loc}}(K, N)$-space if for a cover $\{A_i\}$ of $X$, $A_i\subset X$, $\cup_iA_i=X$, $\op{CD}(K, N)$ holds in each $A_i$.  At the local level $\bigcap_{K'<K}\op{CD}_{\op{loc}}^*(K', N)$ coincide with $\bigcap_{K'<K}\op{CD}_{\op{loc}}(K', N)$. 
We call $(X, d, \mathfrak m)$ is \textbf{essentially non-branching} if for any $\mu, \nu\in \mathcal{P}_2(X)$ with bounded support, each $\pi\in\op{OpGeo}(\mu, \nu)$ is concentrated on a Borel set of non-branching geodesics. It was proved in \cite{CMi}, an essentially non-branching metric measure space $(X, d, \mathfrak m)$ is $\op{CD}(K, N)$ if and only if it is $\op{CD}^*(K, N)$ if and only if it is $\op{CD}^*_{\op{loc}}(K, N)$. By \cite{AGS, AGMR, RS}, a $\RCD(K, \infty)$-space is essentially non-branching.%When $\mathfrak m$ is finite, $\RCD^*(K, N)$ is equivalent to $\RCD(K, N)$.

\begin{Thm}[\cite{EKS}] \label{equivalent}
Assume a metric measure space $(X, d, \mathfrak m)$ with $\op{supp}(\mathfrak m)=X$ satisfies the infinitesimally Hilbertian and Sobolev to Lipschitz property, i.e. any $f\in W^{1,2}(X, d, \mathfrak m)$ with $|\nabla f|_w\leq 1$ $\mathfrak m$-a.e. admits a $1$-Lipschitz representative. Then the followings are equivalence:

(i) $(X, d, \mathfrak m)\in \op{CD}^*(K, N)$;

(ii) The Bakry-Ledoux pointwise gradient estimate $\op{BL}(K, N)$ holds: for $f$ of finite Cheeger energy,
$$|\nabla H_t(f)|^2_w+\frac{4Kt^2}{N(e^{2Kt}-1)}|\Delta H_t f|^2\leq e^{-2Kt} H_t(|\nabla f|^2_w), \mathfrak m-a.e.$$
where 
$$\frac{d}{dt}H_t(f)=\Delta H_t(f), \quad H_0(f)=f.$$

(ii) The Bochner/Bakry-\'Emery inequality $\op{BK}(K, N)$ holds: for $f\in D(\Delta), \Delta f\in W^{1,2}(X, d, \mathfrak m)$, $g\in D(\Delta), g\geq 0$, $\Delta g\in L^{\infty}(X, \mathfrak m)$,
$$\frac12\int \Delta g|\nabla f|^2_w d\mathfrak m-\int g\left<\nabla(\Delta f), \nabla f\right>d\mathfrak m\geq K\int g|\nabla f|^2_wd\mathfrak m +\frac1{N}\int g(\Delta f)^2 d\mathfrak m.$$
\end{Thm}

Last, let's recall the following existence of good cut-off functions.
\begin{Lem}[\cite{MN}]\label{cut-off}
Let $(X, d, \mathfrak m)$ be an $\RCD(K, N)$-space for $N<\infty$ and let compact subset $S=\partial \Omega$ with $\Omega$ Borel. For each $R>0$, $0<10r_1<r_2<R$, there exists a Lipschitz function $\phi: X\to [0, 1]$ such that

(i) $\phi=1$ on $A_{3r_1, \frac{r_2}{3}}(S)$, $\phi=0$ on $X\setminus A_{2r_1, \frac{r_2}{2}}(S)$;

(ii) $r_1^2|\Delta \phi|+r_1|\nabla \phi|\leq C(K, N, R)$ a.e. on $A_{2r_1, 3r_1}(S)$;

(ii) $r_2^2|\Delta \phi|+r_2|\nabla \phi|\leq C(K, N, R)$ a.e. on $A_{\frac{r_2}3, \frac{r_2}2}(S)$.
\end{Lem}

\subsection{Regular Lagrangian flow}
 In this subsection, we recall the definition and some facts about Regular Lagrangian flow (see \cite{AT14, GR18}).

%Consider a time-dependent vector field $V: [0, T]\to L^2(TX)$ where for each $t$, $t\mapsto V(t)\in L^2(TX)$ is Borel. 
\begin{Def}
Consider a metric measure space $(X, d, \mathfrak m)\in \RCD(K, N)$ and a time-dependent vector field $V_t\in L^2([0, 1], L^2_{\op{loc}}(TX))$. We say that
$$F: [0, 1]\times X\to X$$
is a Regular Lagrangian flow (RLF for brief) of $V_t$ if 

(i) $\left(F_s\right)_{\sharp}\mathfrak m\leq C\mathfrak m$ for some $C>0$;

(ii) For $\mathfrak m$-a.e. $x\in X$, the curve $s\mapsto F_s(x), s\in [0, 1]$ is continuous and $F_0(x)=x$;

(iii) For each $f\in W^{1,2}(X, d, \mathfrak m)$, for $\mathfrak m$-a.e. $x\in X$, the function $s\mapsto f(F_s(x))$ belongs to $W^{1,1}(0, 1)$ and 
$$\frac{d}{ds} f(F_s(x))=df (V_s)(F_s(x)), \mathfrak m\times \mathcal L^1-a.e. (x, s).$$
\end{Def}

\begin{Thm}[\cite{AT14}]  \label{exist-RLF}
For a time-dependent vector $V_t\in L^1([0, T], L^2(TX))$ with $V_t\in D(\op{div})$ for a.e. t, if 
\begin{equation*}
\op{div}(V_t)\in L^1([0, T], L^2(X, \mathfrak m)), \max\{-\op{div}(V_t), 0\}\in L^1([0, T], L^{\infty}(X,\mathfrak m)), \nabla V_t\in L^1([0, T], L^2(T^{\otimes2}X)),
\end{equation*}
then there exists a unique, up to $\mathfrak m$-a.e. equality, RLF $(F_t)_{t\in [0, T]}$ for $V_t$ and 
$$(F_t)_{\sharp}\mathfrak m\leq \exp\left(\int_0^t \|\max\{-\op{div}(V_s), 0\}\|_{L^{\infty}}ds\right)d\mathfrak m.$$
\end{Thm}

The RLFs are closely related with the continuity equation. A $W_2$-continuous curve $(\mu_t)_{t\in [0,T]}\subset \mathcal P(X)$ with $\mu_t\leq C\mathfrak m$ and a vector $V_t\in L^2([0, T], L^2(TX))$ is said a solution of the \textbf{continuity equation}
\begin{equation}\frac{d}{dt}\mu_t+\op{div}(V_t\mu_t)=0 \label{continuity-equ}\end{equation}
iff for each $f\in W^{1, 2}(X, d, \mathfrak m)$, the map $t\mapsto \int f d\mu_t$ is absolutely continuous and 
$$\frac{d}{dt}\int f d\mu_t=\int df (V_t)d\mu_t, a.e. t\in [0, T].$$

For $V_t$ as in Theorem~\ref{exist-RLF} and  $F_t$ the unique RLF of $V_t$, let $\mu_t=(F_t)_{\sharp}\mu_0$, $\mu_0\in \mathcal P_2(X)$. 
By \cite{AT14}, $(\mu_t, V_t)$ is a solution of the continuity equation.

\subsection{The differential formula in $\RCD(K, N)$-spaces}

In this subsection, we alway assume $(X, d, \mathfrak m)$ is a $\RCD(K, N)$-space.

Define the class of test functions as 
$$\op{Test}(X)=\{f\in D(\Delta)\cap L^{\infty}(X, \mathfrak m), \, |\nabla f|_w\in L^{\infty}(X, \mathfrak m), \Delta f\in W^{1,2}(X, d, \mathfrak m)\}.$$
It was shown in \cite{Sav14, Han18} that if $f\in \op{Test}(X)$, then $|\nabla f|^2\in D(\Delta)$ and one may define
$$\Gamma_2(f)=\frac12\Delta|\nabla f|^2-\left<\nabla f, \nabla \Delta f\right>.$$

\begin{Def}[\cite{Gi18}]
Let $W^{2,2}(X, d, \mathfrak m)$ be the set of $f\in W^{1,2}(X, d, \mathfrak m)$ satisfying that there is $A\in L^2((T^*)^{\otimes 2}(X))$ such that for any $g_1, g_2, h\in \op{Test}(X)$
$$2\int hA(\nabla g_1, \nabla g_2)d\mathfrak m=-\int \left<\nabla f, \nabla g_1\right>\op{div}(hg_2)+\left<\nabla f, \nabla g_2\right>\op{div}(hg_1)+h\left<\nabla f, \nabla\left<\nabla g_1, \nabla g_2\right>\right>d\mathfrak m.$$
$A$ is called the Hessian of $f$, denoted by $\op{Hess}(f)$.
\end{Def}
It was proved in \cite{Gi18} that $D(\Delta)\subset W^{2,2}(X, d, \mathfrak m)$ and let $H^{2,2}(X)$ be the closure of $D(\Delta)$ in $W^{2,2}(X, d, \mathfrak m)$.

\begin{Thm}[Improved Bocher inequality, \cite{Sav14, Han18}] \label{Boc-ine}
For $(X, d, \mathfrak m)\in \op{RCD}(K, N)$, $K\in \Bbb R, N\in [1, \infty)$ and $f\in \op{Test}(X)$, we have that $f\in W^{2,2}(X, d, \mathfrak m)$ and 
$$\Gamma_2(f)\geq \left(K|\nabla f|^2+|\op{Hess}(f)|_{HS}^2\right)\mathfrak m.$$
\end{Thm}

Given a function $\phi: X\to \Bbb R\cup \{-\infty\}$ not identically $-\infty$, its \textbf{$c$-transform} $\phi^c: X\to \Bbb R\cup \{-\infty\}$ is defined as 
$$\phi^c(x)=\inf_{y\in X}\frac{d^2(x, y)}{2}-\phi(y).$$
We call $\phi$ is \textbf{$c$-concave} if $\phi^{cc}=\phi$. Let $\partial^c\phi\subset X^2$ be the set of $(x, y)\in X^2$ such that
 $$\phi(z)-\phi(x)\leq \frac{d^2(z, y)}2-\frac{d^2(x, y)}2, \, \forall\, z\in X.$$

A test plan $\pi\in \op{OpGeo}(\mu, \nu)$ if and only if there is a $c$-concave function $\phi$ such that $\op{supp}((e_0, e_1)_{\sharp}\pi)\subset \partial ^c\phi$ and such $\phi$ is called \textbf{Kantorovich potential} from $\mu$ to $\nu$. And for any $t\in (0, 1)$, $t\phi$ is a Kantorovich potential from $\mu$ to $(e_t)_{\sharp}(\pi)$.

\begin{Thm}[\cite{Gi13}, see also \cite{GiT}]\label{first-diff1}
Consider a $\RCD(K, N)$-space $(X, d, \mathfrak m)$. Let $\pi$ be an optimal geodesic test plan with bounded support
%Let $\mu_t$ be a $W_2$-geodesic with compact support and $\mu_t\leq C\mathfrak m$, $t\in[0, 1]$ 
and consider $f\in W^{1,2}(X, d, \mathfrak m)$. We have that the map $t\mapsto f\circ e_t$ is $C^1([0, 1])$ and for $t\in [0, 1]$
\begin{equation}
\frac{d}{dt}f\circ e_t =\left<\nabla f, \nabla \phi_t\right>\circ e_t,
\end{equation}
where $\phi_t$ is any function such that for some $s\neq t$, $s\in [0, 1]$, $-(s-t)\phi_t$ is a Kantorovich potential from $\mu_t=(e_t)_{\sharp}(\pi)$ to $\mu_s=(e_t)_{\sharp}(\pi)$.
\end{Thm}

\begin{Thm}[\cite{GiT}]
Let the assumption be as in Theorem~\ref{first-diff1} and let $f\in H^{2,2}(X)$. Then the map $t\mapsto f\circ e_t$ is in $C^2([0,1])$ and for $t\in [0, 1]$
\begin{equation}
\frac{d^2}{dt^2}f\circ e_t=\op{Hess}(f)(\nabla \phi_t, \nabla \phi_t)\circ e_t.
\end{equation}
\end{Thm}

A corollary of Theorem~\ref{first-diff1} is the following first order differential estimates along geodesics (see \cite[Corollary 3.14]{Deng}):
\begin{Cor}[\cite{Deng}] \label{first-diff2}
Let $(X, d, \mathfrak m)$ be a $\RCD(K, N)$-space and let $p\in X$, $f\in W^{1, 2}(X, d, \mathfrak m)$. For $\mathfrak m$-a.e. $x\in X$, the map $t\mapsto f(\gamma_{x, p}(t))$ is in $W^{1,1}_{\op{loc}}([0, d(p, x)))$ and 
\begin{equation}
\frac{d}{dt} f(\gamma_{x, p}(t))= -df(\nabla d_p)(\gamma_{x, p}(t)), \text{ for } a.e. t\in [0, d(p, x)),
\end{equation}
where $\gamma_{x, p}$ is a unit speed geodesic from $x$ to $p$. 
\end{Cor}

For a fixed point $p\in X$ and a measure $\mu\in \mathcal P_2(X)$ with $\mu\leq C\mathfrak m$, let $\gamma_{x, p}(t), t\in [0, 1]$ be a constant speed geodesic from $x\in \op{supp}(\mu)$ to $p$, let $D=\sup _{x\in \op{supp}(\mu)}d(x, p)$ and let $\mu_t=(\gamma_{\cdot, p}(t))_{\sharp} (\mu)$.  By \cite[Theorem 3.15]{Deng}, we have that $(\mu_t)_{t\in [0, 1-\delta]}\leq C(K, N, D, \delta)\mathfrak m$ is a $W_2$-geodesic for any $1>\delta>0$, and $(\mu_t)_{t\in [0, 1-\delta]}$ solves the continuity equation
$$\frac{d}{dt} \mu_t+\op{div}(-\nabla d_p \mu_t)=0.$$
And we have the second order differential formula, as \cite[Proposition 3.21]{Deng},
\begin{Cor} \label{second-diff2}
Let the assumption be as in Corollary~\ref{first-diff2}, let $\Pi\leq C(\mathfrak m\times \mathfrak m)$ be a nonnegative, compactly supported measure on $X\times X$ and let $f\in H^{2,2}(X)$. Then for each $t, s\in (0, 1]$, 
\begin{equation}
\int \left<\nabla f, \nabla d_x\right>(\tilde{\gamma}_{x, y}(t))-\left<\nabla f, \nabla d_x\right>(\tilde{\gamma}_{x, y}(s))d\Pi(x, y)=\int_s^t\int d(x, y)\op{Hess}(f)(\nabla d_x, \nabla d_x)(\tilde{\gamma}_{x, y}(\tau))d\Pi(x, y)d\tau,
\end{equation}
where $\tilde{\gamma}_{x, y}:[0, 1]\to X$ is a constant speed geodesic from $x$ to $y$.
\end{Cor}

\subsection{Warped product and $(K, N)$-cone over metric measure spaces}

%In this subsection, we will state some definitions and results of warped product and $(K, N)$-cones over metric measure spaces. See \cite{Ket1} for more details.

Consider two metric measure spaces $(Z, d_Z, \mathfrak m_Z)$ and $(Y, d_Y, \mathfrak m_Y)$ and two continuous maps $w_d, w_m: Z\to [0, \infty)$ with $\{w_d=0\}\subset \{w_m=0\}$. A \textbf{warped product} $Z\times_w Y$ is the space $Z\times Y$ admitting the metric 
$$d_w(p, q)=\inf\{l_w(\gamma),\, \gamma \text{ is a absolutely continuous curve between } p, q \},$$
where $\gamma=(\gamma^Z, \gamma^Y)$,
$$l_w(\gamma)=\int_0^1\sqrt{|\dot \gamma_t^Z|^2+w_d^2(\gamma_t^Z)|\dot \gamma_t^Y|^2}dt.$$
And the measure $\mathfrak m_w$ on $Z\times_w Y$ is defined as 
$$d\mathfrak m_w= w_md\mathfrak m_Z\otimes d\mathfrak m_Y.$$

\begin{Def}
The $(K, N)$-cone $(C(Y), d_{K}, \mathfrak m_N)$ over a metric measure space $(Y, d, \mathfrak m_Y)$ is defined as follows: let $H=K/(N-1)$ and $(t, y_1), (t, y_2)\in C(Y)$,

$$C(Y)=\begin{cases} \left[0, \pi/\sqrt H\right]/ \{\{0, \pi/\sqrt H\}\times Y\}, & K>0;\\ [0, \infty)\times Y/\{0\times Y\}, & K\leq 0; \end{cases}$$
$$d_K((t, y_1), (s, y_2))=\begin{cases} \sqrt{t^2+s^2-2st\cos\min\{\pi, d(y_1, y_2)\}} & K=0;\\ 
(\op{sn}'_H)^{-1}\left(\op{sn}'_H(t)\op{sn}'_H(s)+H\op{sn}_H(s)\op{sn}_H(t)\cos \min\{\pi, d(y_1, y_2)\}\right), & K\neq 0; \end{cases}$$

$$\mathfrak m_N=\op{sn}^{N-1}_H(t)dt\otimes \mathfrak m_Y.$$
\end{Def}

It is obvious that if $\op{diam}(Y)\leq \pi$, then $C(Y)= I_K\times_{\op{sn}_H}^{N-1}Y$, where $I_K=[0, \frac{\pi}{\sqrt H}]$ for $K>0$ and $I_K=[0, \infty)$ for $K\leq 0$ and $w_m=\op{sn}_H^{N-1}$, $w_d=\op{sn}_H$. And in \cite{Ket1}, Ketterer showed that for $N\geq 2$, $K\geq 0$, a $(K, N)$-cone over $(Y, d, \mathfrak m)$ is a $\RCD(K, N)$-space if and only if $(Y, d, \mathfrak m)$ is a $\RCD(N-2, N-1)$-space. More precisely,

\begin{Thm}[\cite{Ket1}] \label{cone-rcd}
(i) Assume  a metric measure space $(Y, d, \mathfrak m)\in \RCD(N-2, N-1)$, $N\geq 2$, $K\geq 0$ and $\op{diam}(Y)\leq \pi$, Then the $(K, N)$-cone over $(Y, d, \mathfrak m)$, $(C(Y), d_K, \mathfrak m_N)\in \RCD(K, N)$;

(ii) Assume the  $(K, N)$-cone over a metric measure space $(Y, d, \mathfrak m)$, $(C(Y), d_K, \mathfrak m_N)\in \RCD(K, N)$ with $N\geq 2$, then $(Y, d, \mathfrak m)\in \RCD(N-2, N-1)$ with $\op{diam}(Y)\leq \pi$.
\end{Thm}

\subsection{Measure decomposition and Heintze-Karcher inequality in $\RCD(K, N)$-spaces}

In this subsection, we briefly recall the following measure decomposition in $\RCD(K, N)$-spaces by \cite{CM17, CMo} and the Heintze-Karcher inequality in $\RCD(K, N)$-spaces given by \cite{Ket2}. Note that in \cite{CM17}, to derive the measure decomposition result they also assume the condition that $\mathfrak m(X)<\infty$. In \cite{CMo}, they showed that this assumption is unnecessary. 

Consider a measurable space $(R, \mathcal R, \mathfrak m)$ and a map $\mathfrak Q: R\to Q$. One can equip $Q$ with a $\sigma$-algebra $\mathcal Q$: $B\in \mathcal Q$ iff $\mathfrak Q^{-1}(B)\in \mathcal R$. Let $\mathfrak Q_{\sharp} \mathfrak m=\mathfrak q$ which is a probability measure on $Q$.

A \textbf{disintegration} of $\mathfrak m$ which is consistent with $\mathfrak Q$ is a map $\mathcal R\times Q\to [0, 1]$, $(A, \alpha)\mapsto \mathfrak m_{\alpha}(A)$ satisfying that:

(i) $\mathfrak m_{\alpha}$ is a probability measure on $(R, \mathcal R)$ for each $\alpha\in Q$;

(ii) $\alpha\mapsto \mathfrak m_{\alpha}(A)$ is $\mathfrak q$-measurable for each $A\in \mathcal R$;

(iii) For $A\in \mathcal R, B\in \mathcal Q$, 
$$\mathfrak m(A\cap \mathfrak Q^{-1}(B))=\int_B \mathfrak m_{\alpha}(A)d\mathfrak q(\alpha).$$

We call $\{\mathfrak m_{\alpha}\}_{\alpha\in Q}$ a disintegration of $\mathfrak m$ and call $\mathfrak m_{\alpha}$ the conditional probability measures.

A  disintegration $\{\mathfrak m_{\alpha}\}_{\alpha\in Q}$ is strongly consistent with $\mathfrak Q$ if for $\mathfrak q$-a.e. $\alpha$, $\mathfrak m_{\alpha}(\mathfrak Q^{-1}(\alpha))=1$.

Let $(X, d, \mathfrak m)$ be as in the beginning of this section and  be essentially non-branching.  Let  $u: X\to \Bbb R$ be a $1$-Lipschitz map.
A set $A\subset X\times X$ is $d$-cyclically monotone if for any finite set of points $(x_1, y_1), \cdots, (x_n, y_n)\in A$
$$\sum_{i=1}^nd(x_i, y_i)\leq \sum_{i=1}^n d(x_i, y_{i+1}), \, y_{n+1}=y_1.$$
A $d$-cyclically monotone set associated with $u$ is defined as
$$\Gamma=\{(x, y)\in X\times X, \, u(x)-u(y)=d(x, y)\}.$$ 
It is obvious that $(x, y)\in \Gamma$ implies that for any minimizing geodesic $\gamma$ from $x$ to $y$, $(\gamma_t, \gamma_s)\in \Gamma$ for $0\leq s\leq t\leq 1$. 

Let the transport rays $R=\Gamma\cup \Gamma^{-1}$, where $\Gamma^{-1}=\{(x, y)\in X\times X, \, (y, x)\in \Gamma\}$. Let $\mathcal T=P_1(R\setminus \{(x, y), x=y\in X\})\subset X$ where $P_1(x, y)=x$.  And the sets
$$A_{+}=\{x\in \mathcal T, \, \exists z, w\in \Gamma(x), (z, w)\notin R\},$$
$$A_{-}=\{x\in \mathcal T, \, \exists z, w\in \Gamma^{-1}(x), (z, w)\notin R \},$$ 
 are called forward and backward branching points respectively,
where $\Gamma(x)=\{y\in X, (x, y)\in \Gamma\}$, and similar define $\Gamma^{-1}(x), R(x)$.
The initial and final points are
$$\mathfrak a=\{x, \, \Gamma^{-1}(x)=\{x\} \}, \,\, \mathfrak b=\{x, \, \Gamma(x)=\{x\}\}.$$

  Let  $\mathcal T_u=\mathcal T\setminus (A_{+}\cup A_{-})$. It was proved that
 $$x\sim y \Leftrightarrow (x, y)\in R$$ is an equivalence relation on $\mathcal T_u$ (\cite{CM17, CMo}). Denote this relation by $R_u$ and  the equivalence classes by $\{X_{\alpha}\}_{\alpha\in Q}$. 
 Each $X_{\alpha}$ is isometric to an interval $I_{\alpha}\subset \Bbb R$ via an isometry $\gamma_{\alpha}: I_{\alpha}\to X_{\alpha}$ such that $d(\gamma_{\alpha}(t), \gamma_{\alpha}(s))=|t-s|$ for $t, s\in I_{\alpha}$  .
 The map $\gamma_{\alpha}$ extends to a geodesic in $X$ which is also denoted by $\gamma_{\alpha}$. Denote the closure $\bar I_{\alpha}$ of $I_{\alpha}$ by $[a(X_{\alpha}), b(X_{\alpha})]$.

\begin{Thm}[\cite{CM17, CMo}]\label{dis-com}
Let $(X, d, \mathfrak m)$ be an essentially non-branching $\op{CD}(K, N)$-space with $\op{supp}(\mathfrak m)=X$. Let $u: X\to \Bbb R$ be a $1$-Lipschitz function. Then 

(i) there exists a disintegration $\{\mathfrak m_{\alpha}\}_{\alpha\in Q}$ of $ \mathfrak m\llcorner_{\mathcal T_u}$ that is strongly consistent;

(ii) there is $Q'\subset Q$ such that $\mathfrak{q}(Q\setminus Q')=0$ and for each $\alpha\in Q'$, $\mathfrak m_{\alpha}$ is a Radon measure with $\mathfrak m_{\alpha}=h_{\alpha}\mathcal{H}^1\llcorner_{X_{\alpha}}$ and $(X_{\alpha}, d, \mathfrak m_{\alpha})$ verifies the condition $\op{CD}(K, N)$. More precisely, for $\alpha\in Q'$, it holds that
\begin{equation}h_{\alpha}(\gamma(t))^{\frac1{N-1}}\geq \sigma_{K, N-1}^{1-t}(|\gamma'|)h_{\alpha}(\gamma(0))^{\frac1{N-1}}+\sigma_{K, N-1}^{t}(|\gamma'|)h_{\alpha}(\gamma(1))^{\frac1{N-1}},\label{ele-com}\end{equation}
for each constant speed geodesic $\gamma: [0, 1]\to (a(X_{\alpha}), b(X_{\alpha}))$.
\end{Thm}
And by \cite[Corollary 4.3]{Ket2}, we have that
\begin{Lem}[\cite{Ket2}] \label{ele-com0}
Let the assumption be as in Theorem~\ref{dis-com}. Then for each $0\leq a<b\leq 1$,
$$\frac{h_{\alpha}(b)}{h_{\alpha}(a)}\leq \left(\op{sn}'_H(b-a)+\frac{(\ln h_{\alpha})_+'(a)}{N-1}\op{sn}_H(b-a)\right)_+^{N+1},$$
where 
$$(\ln h_{\alpha})_+'(a)=\lim_{h\downarrow 0}\frac{\ln h_{\alpha}(a+h)-\ln h_{\alpha}(a)}{h},$$
$h_{\alpha}(t)=h_{\alpha}(\gamma(t))$ and $f_+=\max\{f, 0\}$.
\end{Lem}

Consider $(X, d, \mathfrak m)$ be as above lemma. Let $\Omega\subset X$ be a Borel subset and let $S=\partial \Omega$ with $\mathfrak m(S)=0$. Let 
$$d_s=\begin{cases} d(x, S), & x\in X\setminus \Omega;\\ -d(x, S), & x\in \Omega. \end{cases}$$
Then $d_s$ is $1$-Lipschitz. By Theorem~\ref{dis-com}, there is a disintegration of $d_s$, $\{\mathfrak m_{\alpha}\}_{\alpha\in Q}$ and a partition $\{X_{\alpha}\}_{\alpha\in Q}$ of $X$ up a measure zero set (By \cite[Lemma 3.4]{CMo}, $\mathfrak m(\mathcal T\setminus \mathcal T_{d_s})=0$. And $\mathcal T\supset X\setminus S$. Thus $\mathfrak m(S)=0$ implies $\mathfrak m(X\setminus \mathcal T_{d_s})=0$).
%And the transport set of $d_s$,$\mathcal T\supset X\setminus S$,  $\mathfrak m(X\setminus \mathcal T_{d_s})=0$.

Let
$A= \mathfrak{Q}^{-1}(\mathfrak{Q}(S\cap \mathcal T_{d_s}))$. Then for each $\alpha\in \mathfrak Q(A)$, there is a unique $t_{\alpha}\in (a(X_{\alpha}), b(X_{\alpha}))$ such that $X_{\alpha}\cap S=\{\gamma(t_{\alpha})\}\neq \emptyset$. Identify the measurable set $\mathfrak Q(A)\subset Q$ with $A\cap S$ and one can assume $t_{\alpha}=0$. 

Let $\hat Q=\{\alpha\in Q, \, \overline{X_{\alpha}}\setminus \mathcal T_{u}\subset \mathfrak a\cup \mathfrak b\}$. By \cite[Theorem 7.10]{CM17}, $\mathfrak q(Q\setminus \hat Q)=0$. Let  $\mathcal T^*_{d_s}=\mathfrak Q^{-1}(\hat Q\cap Q')$. The sets $B_{\op{in}}=\Omega^{\circ}\cap \mathcal T^*_{d_s}\setminus (A\cap \mathcal T^*_{d_s})$ and 
 $B_{\op{out}}=\Omega^{c}\cap \mathcal T^*_{d_s}\setminus (A\cap \mathcal T^*_{d_s})$
are measurable (see \cite[Remark 5.1]{Ket2}).

We say that $S$ has \textbf{finite outer curvature} if $\mathfrak m(B_{\op{out}})=0$, $S$ has finite inner curvature if $\mathfrak m(B_{\op{in}})=0$, and $S$ has \textbf{finite curvature} if $\mathfrak m(B_{\op{in}}\cup B_{\op{out}})=0$.

If $S$ has finite outer curvature, we can define its \textbf{outer mean curvature} as
$$p\in S\mapsto H^{+}(p)=\left\{\begin{array}{cc}
\frac{d^{+}}{dr}\ln h_{\alpha}(\gamma_{\alpha}(0)), & p=\gamma_{\alpha}(0)\in S\cap A\cap \mathcal T^{*}_{d_s}\\
-\infty, & p\in R_{d_s}(B_{\op{in}})\cap S,\\
c (\text{for some } c\in \Bbb R), & \text{ otherwise}.\end{array}\right.$$
Switch the roles of $\Omega$ and $\overline{\Omega^c}$ and assume $S$ has finite inner curvature, we call the corresponding outer mean curvature the \textbf{inner mean curvature} and write as $H^{-}$.

If $S$ has finite curvature, the \textbf{mean curvature} defined as $\max\{H^+, -H^-\}=m$.

The \textbf{surface measure} of $S$ is defined as 
$$\int_S \phi(x)d\mathfrak m_S(x)=\int_{\mathfrak Q(A\cap \mathcal T^*_{d_s})}\phi(\gamma_{\alpha}(0))h_{\alpha}(0)d\mathfrak q(\alpha),$$
for any continuous function $\phi: X\to \Bbb R$.

In \cite{Ket2}, Ketterer generalized the Heintze-Karcher inequality to $\op{CD}(K, N)$-spaces.
\begin{Thm} \label{HK-general}
Assume $(X, d, \mathfrak m)\in \op{CD}(K, N)$ is an essentially non-branching metric measure space, $N>1$. Let $\Omega\subset X$ be a closed Borel subset and let $S=\partial \Omega$ such that $\mathfrak m(S)=0$ and $S$ has finite outer curvature. Then
$$\mathfrak m(B_t(\Omega)\setminus \Omega)\leq \int_S\int_0^t J_{K, H^+(p), N}(r)dr d\mathfrak m_S(p), \forall \, t\in (0, \tilde D],$$
where $\tilde D=\op{diam}(X)$, $J_{K, H, N}(r)=\left(\op{sn}'_{K/N-1}(r)+\frac{H}{N-1}\op{sn}_{K/N-1}(r)\right)_+^{N-1}$.

If $S$ has finite curvature, then
\begin{equation}\mathfrak m(X)\leq \int_S\int_{-\tilde D}^{\tilde D} J_{K, m(p), N}(r)dr d\mathfrak m_S(p).\label{vol-1}\end{equation}

In particular, if $(X, d, \mathfrak m)$ is a $\op{RCD}(K, N)$-space, $K>0$, then the equality holds in \eqref{vol-1} iff there is a $\op{RCD}(N-2, N-1)$-space $(Y, d_Y, \mathfrak m_Y)$ such that  $(X, d, \mathfrak m)=(C(Y), d_K, \mathfrak m_N)$
and $S$ is a constant mean curvature surface in $X$.
\end{Thm}

\section{Laplacian estimates}

Let $(X_i, d_i, \mathfrak m_i)$ be a sequence of $\RCD(K, N)$-spaces and let $\Omega_i\subset X_i$ be a Borel subset with $S_i=\partial \Omega_i$ closed and $\op{diam}(S_i)\leq D$. Define a signed distance function associated to $\Omega_i$.
$$d_{s, i}(x)=\begin{cases} d_i(x, S_i), & x\in X_i\setminus \Omega_i;\\ -d_i(x, S_i), & x\in \Omega_i. \end{cases}$$
Obviously, $d_{s, i}$ is $1$-Lipschitz.  

Assume $(X_i, d_i, \mathfrak m_i)$ is measured Gromov-Hausdorff convergent to a metric measure space $(X, d, \mathfrak m)\in \RCD(K, N)$.
Then by \cite[Proposition 2.70]{Vi} or \cite[Proposition 2.12]{MN}, there is a $1$-Lipschitz function $d_s: X\to \Bbb R$ such that 
$d_{s, i}$ converges uniformly to $d_s$ on any compact set.  Let $S=\{x\in X, \, d_s(x)=0\}$, $\Omega=\{x\in X, \, d_s(x))\leq 0\}$. Then as in \cite[Lemma 3.25]{Hu}, we know that $d_s$ is a signed distance function associated to $\Omega$. Let $H=\frac{K}{N-1}$ and for $m\neq 0$ or $K\neq 0$, let 
$$m=(N-1)\frac{\op{sn}'_H(r_0)}{\op{sn}_H(r_0)}.$$

In this section, we will show that
\begin{Thm}[Laplacian estimates] \label{lap-main}
Let $(X, d, \mathfrak m)$, $d_s$ be as above. Assume $S_i$ has finite outer curvature, $\op{diam}(S_i)\leq D$ and the mean curvature $m_i(x_i)\leq m$ for each $i$ and any $x_i\in S_i$.
Then if \eqref{lqvol-com}
$$\frac{\mathfrak m(A^+_{a, b}(S_i))}{\mathfrak m(S_i)}\geq (1-\epsilon_i) \int_a^b \left(\op{sn}'_H(r)+ \frac{m}{N-1}\op{sn}_H(r)\right)^{n-1}dr,$$
holds with $\epsilon_i\to 0$, $d_s\in \Delta(A_{a, b}(S))$.  In particular, for $x\in A_{a, b}(S)$,

(i) For $m=0$ and $K=0$, we have that $\Delta d_s=0$;

(ii) For $m\neq 0$ or $K\neq 0$, $$\Delta d_s=(N-1)\frac{\op{sn}'_H(d_s+r_0)}{\op{sn}_H(d_s+r_0)}.$$
\end{Thm}

%by assuming \eqref{lqvol-com}, we will give Laplacian estimates of $d_s$ in two cases. One is for $m=0, K=0$ corresponding to the quantitative splitting case and one is for $m\neq 0$ or $m=0$, $K>0$ corresponding to the almost metric cone case.

To prove the above Laplacian estimates, we first recall the following Laplacian formula in \cite[Corollary 4.16]{CMo}.
\begin{Thm}[\cite{CMo}]
\label{lap-com}
Let $(X, d, \mathfrak m)$ be a $\RCD(K, N)$-space. Consider the signed distance function $d_s$ associated with a Borel subset $\Omega\subset X$ and a compact boundary $S=\partial \Omega$ with $\mathfrak m(S)=0$. And assume the associated disintegration of $d_s$, $\mathfrak m=\int_Q\int_{X_{\alpha}} h_{\alpha}(r)dr d\mathfrak q(\alpha)$. Then $d_s\in D(\Delta, X\setminus S)$ and 
$$\Delta d_s\llcorner_{X\setminus S}=\left(\ln h_{\alpha}\right)' \mathfrak m\llcorner_{X\setminus S}-\int_Q h_{\alpha}\left(\delta_{a(X_{\alpha})\cap (X\setminus \Omega)}+\delta_{b(X_{\alpha})\cap \Omega}\right)d\mathfrak q(\alpha),$$
where $\left(\ln h_{\alpha}\right)'$ is roughly the directional derivative of $\ln h_{\alpha}$ in the direction of $\nabla d_s$.
In particular 
$$[\Delta d_s\llcorner_{X\setminus S}]_{\op{reg}}=\left(\ln h_{\alpha}\right)' \mathfrak m\llcorner_{X\setminus S}$$ and 
$$[\Delta d_s\llcorner_{X\setminus S}]_{\op{sing}}=-\int_Q h_{\alpha}\left(\delta_{a(X_{\alpha})\cap (X\setminus \Omega)}+\delta_{b(X_{\alpha})\cap \Omega}\right)d\mathfrak q(\alpha).$$
\end{Thm}

Note that in \cite[Corollary 4.16]{CMo}, there is a negative sign in front of  the derivative $(\ln h)'_{\alpha}$, where they defined $h'_{\alpha}$ as 
$$h'_{\alpha}(x)=\lim_{t\to 0}\frac{h_{\alpha}(g_t(x))-h_{\alpha}(x)}{t},$$
and $g_t(x)=y$ such that $d_s(x)-d_s(y)=t$. Roughly speaking the derivative $h'_{\alpha}$ there is the directional derivative in the direction of $-\nabla d_s$. Compared the one in Lemma~\ref{ele-com0} in this paper, we always denote
$$h'_{\alpha}(x)=\lim_{t\to 0}\frac{h_{\alpha}(\gamma(t_0+t))-h_{\alpha}(\gamma(t_0))}{t},$$
where $\gamma(t_0)=x$ and $\gamma$ is a unit speed geodesic in $X_{\alpha}$ such that $d_s(\gamma(l))-d_s(\gamma(t))=l-t$.

Now using \eqref{ele-com}, as the discussion of \cite[Lemma 4.1]{Ket2} and the proof of Laplacian comparison in manifolds with lower Ricci curvature bound, we have the following Laplacian comparison.
\begin{Lem} \label{glap-com}
Let the assumption be as in Theorem~\ref{lap-com} and assume $S$ has finite outer curvature. Assume that for each $x\in S$, the mean curvature $$m(x)\leq m.$$

(i) If $m=0$ and $K= 0$, we have 
$$[\Delta d_s\llcorner_{X\setminus (\Omega\cup S)}]_{\op{reg}}(x)\leq 0;$$

(ii) If $m\neq 0$ or $K\neq 0$, then
$$[\Delta d_s\llcorner_{X\setminus (\Omega\cup S)}]_{\op{reg}}(x)\leq (N-1)\frac{\op{sn}'_{H}(d_s(x)+r_0)}{\op{sn}_{H}(d_s(x)+r_0)}.$$
\end{Lem}
\begin{proof}
Let $u(t)=h^{\frac1{N-1}}_{\alpha}(\gamma(t))$ where $h_{\alpha}$ satisfies \eqref{ele-com} as in Theorem~\ref{dis-com}. Then $u$ is semi-concave and satisfies
$$u''+Hu\leq 0$$
in the distributional sense. And the limits
$$u'_+(r)=\lim_{\delta\downarrow 0}\frac{u(r+\delta)-u(r)}{\delta}, \quad u'_-(r)=\lim_{\delta\downarrow 0}\frac{u(r-\delta)-u(r)}{-\delta}$$
exits. 

Take $\phi\in C_0^{\infty}((-1, 1))$, $\int_{-1}^1\phi=1, \phi_{\epsilon}(t)=\frac1{\epsilon}\phi(\frac{t}{\epsilon})$ and let 
$$\tilde u(s)=\int_{-\epsilon}^{\epsilon}\phi_{\epsilon}(-r)u(s-r)dr.$$
Then 
$$\tilde u''(s)\leq -H\tilde u,$$
and 
$$\left(\frac{\tilde u'}{\tilde u}\right)'=\frac{\tilde u''}{\tilde u}-\left(\frac{\tilde u'}{\tilde u}\right)^2\leq -H-\left(\frac{\tilde u'}{\tilde u}\right)^2.$$
Let $f=\frac{\tilde u'}{\tilde u}$ and let $f_H=\frac{\op{sn}'_H}{\op{sn}_H}$. Then $f'_H=-H-f^2_H$ and 
\begin{eqnarray*}
\left(\op{sn}_H^2(s)(f(s)-f_H(s))\right)' & =& 2\op{sn}_H(s)\op{sn}'_H(s)(f(s)-f_H(s))+\op{sn}_H^2(s)(f'(s)-f'_H(s))\\
&\leq & 2\op{sn}^2_H(s) f_H(s)(f(s)-f_H(s))-\op{sn}^2_H(s)(f^2(s)-f_H^2(s))\\
&=& -\op{sn}_H^2(s)(f(s)-f_H^2(s))^2\leq 0.
\end{eqnarray*}
Thus for $s\geq 0$,
$$\op{sn}_H^2(s+r_0)(f(s+r_0)-f_H(s+r_0))-\op{sn}_H^2(r_0)(f(r_0)-f_H(r_0))\leq 0.$$
If $f(r_0)\leq f_H(r_0)$, then
$$f(s+r_0)\leq f_H(s+r_0).$$

Now note that as $\epsilon\to 0$, $\tilde u\to u$ and $\tilde u'_+\to u'_+$ and $\frac{u'}{u}=\frac1{N-1}\left(\ln h_{\alpha}\right)'$. And by the Laplacian formula Theorem~\ref{lap-com}, we have (ii). 

For (i), as above we have that  $f'\leq 0$. Thus for $s\geq 0$,
$$f(s+r_0)\leq f(r_0)\leq 0.$$
\end{proof}

\begin{Rem} The above laplacian comparison can also be seen in \cite{BKMW}.
\end{Rem}

Now by Lemma~\ref{glap-com} and Lemma~\ref{ele-com0}, we have the following volume element comparison.  

\begin{Lem}[Volume element comparison] \label{vol-ele-com}
Let the assumption be as in Lemma~\ref{glap-com}.  Then for any $r_1, r_2\in (0, b(X_{\alpha}))$, $r_1\leq r_2$,

(i) for $m=0$ and $K=0$, 
$$h_{\alpha}(r_2)\leq h_{\alpha}(r_1);$$

(ii) for $m\neq 0$ or $K\neq 0$,
$$\frac{h_{\alpha}(r_2)}{\op{sn}_H^{N-1}(r_2+r_0)}\leq \frac{h_{\alpha}(r_1)}{\op{sn}_H^{N-1}(r_1+r_0)}.$$
\end{Lem}
\begin{proof}
By Lemma~\ref{ele-com0},  for $r_1\leq r_2$ as above,
\begin{eqnarray*}
h_{\alpha}(r_2)&\leq & J_{K, H^+(\gamma_{\alpha}(r_1)), N}(r_2-r_1)h_{\alpha}(r_1)\\
&= & \left(\op{sn}'_{H}(r_2-r_1)+\frac{\left(\ln h_{\alpha}\right)'_+(r_1)}{N-1}\op{sn}_{H}(r_2-r_1)\right)_+^{N-1} h_{\alpha}(r_1).
\end{eqnarray*}
Thus for $m=0$ and $K=0$, by Lemma~\ref{glap-com}, $(\ln h_{\alpha})'\leq 0$, we have 
$$h_{\alpha}(r_2)\leq h_{\alpha}(r_1);$$
For $m\neq 0$ or $K\neq 0$, by Lemma~\ref{glap-com},
$$h_{\alpha}(r_2)\leq \frac{\op{sn}_H^{N-1}(r_2+r_0)}{\op{sn}_H^{N-1}(r_1+r_0)}h_{\alpha}(r_1).$$
\end{proof}

In the following, let $h_{\alpha}(r)=0$ when $r$ increases and $h_{\alpha}(r)$ becomes undefined. If $S$ has finite outer curvature, then for almost all $a\geq 0$,
$$\mathfrak m(\partial B_a(S))=\limsup_{\delta\to 0}\frac{\mathfrak m(A_{a, a+\delta}(S))}{\delta}=\int_{\mathfrak Q(\partial B_a(S)\cap \mathcal T^*_{d_s})} h_{\alpha}(a)d\mathfrak q(\alpha).$$
Using above lemma, we have the  following relative volume comparison. 

\begin{Lem}[Relative volume comparison]\label{rel-vol}
Let the assumption be as in Lemma~\ref{glap-com}. Then for $0\leq a<b$, 

(i) if $m=0$ and $K=0$, then
\begin{equation}\mathfrak m(\partial B_b(S)\leq \mathfrak m(\partial B_a(S));\label{sph-com0}\end{equation}
\begin{equation}\mathfrak m(A_{a, b}(S))\leq (b-a)\mathfrak m(\partial B_a(S));\label{ann-sph0}\end{equation}

(ii) if $m\neq 0$ or $K\neq 0$,
\begin{equation}
\mathfrak m(\partial B_b(S))\leq \mathfrak m(\partial B_a(S))\frac{\svolsp{H}{b+r_0}}{\svolsp{H}{a+r_0}}, \label{sph-com}
\end{equation}
\begin{equation}
\mathfrak m(A_{a, b}(S))\leq \mathfrak m(\partial B_a(S))\frac{\svolann{H}{a+r_0, b+r_0}}{\svolsp{H}{a+r_0}}. \label{ann-sph}
\end{equation}
\end{Lem}

\begin{proof}
The proof is similar as the the one of \cite[Theorem 1.1]{Ket2},
\begin{eqnarray*}
\mathfrak m(A_{a, b}(S))&=& \mathfrak m(A_{a, b}(S)\cap \mathcal T^*_{d_s}\setminus B_{\op{out}})=\int_{\mathfrak Q(A_{a, b}(S)\cap \mathcal T^*_{d_s})}\int_{A_{a, b}(S)\cap X_{\alpha}}h_{\alpha}(r)dr d\mathfrak q(\alpha)
%&\leq & \int_{\mathfrak Q(A_{a, b}(S)\cap \mathcal T^*_{d_s})}\int_a^b J_{K, H^+(\gamma_{\alpha}(a)), N}(r-a)h_{\alpha}(a)dr d\mathfrak q(\alpha)\\
%&\leq &  \int_{\mathfrak Q(A_{a, b}(S)\cap \mathcal T^*_{d_s})}\int_a^b \left(\op{sn}'_{H}(r-a)+\frac{\left(\ln h_{\alpha}\right)'(\gamma_{\alpha}(a))}{N-1}\op{sn}_{H}(r-a)\right)_+^{N-1}dr h_{\alpha}(a)d\mathfrak m_S(p).
\end{eqnarray*}

By Lemma~\ref{vol-ele-com}, for $m\neq 0$ or $K\neq 0$, %and $m=(N-1)\frac{\op{sn}'_{H}(r_0)}{\op{sn}_{H}(r_0)}$, then 
%$$(\ln h_{\alpha})'(\gamma_{\alpha}(a))\leq (N-1)\frac{\op{sn}'_H(a+r_0)}{\op{sn}_H(a+r_0)}$$
%and thus
\begin{eqnarray*}
\mathfrak m(A_{a, b}(S))& \leq &  \int_{\mathfrak Q(A_{a, b}(S)\cap \mathcal T^*_{d_s})}\int_a^b \left(\frac{\op{sn}_{H}(r+r_0)}{\op{sn}_{H}(a+r_0)}\right)^{N-1}drh_{\alpha}(a) d\mathfrak q(\alpha)\\
& \leq & \mathfrak m(\partial B_a(S))\int_a^b \left(\frac{\op{sn}_{H}(r+r_0)}{\op{sn}_{H}(a+r_0)}\right)^{N-1}dr\\
&=& \mathfrak m(\partial B_a(S))\frac{\svolann{H}{a+r_0, b+r_0}}{\svolsp{H}{a+r_0}};
\end{eqnarray*}
For $m=0$ and $K=0$, %then $(\ln h_{\alpha})'(\gamma_{\alpha}(a))\leq 0$ and thus
$$\mathfrak m(A_{a, b}(S))\leq (b-a) \mathfrak m(\partial B_a(S)).$$
\end{proof}

By above relative volume comparison, as in \cite{CC1}, we have that
\begin{Lem} \label{qua-com}
Let the assumption be as in Lemma~\ref{rel-vol} and assume \eqref{lvol-com}. Then for each $d\in (a, b)$, 

(i) if $m=0$ and $K=0$,
\begin{equation}
\frac{\mathfrak m(\partial B_d(S))}{\mathfrak m(\partial B_a(S))}\geq 1-\epsilon\frac{b-a}{b-d};\label{csph-com0}
\end{equation}

(ii) if $m\neq 0$ or $K\neq 0$,
\begin{equation}
\frac{\mathfrak m(\partial B_d(S))}{\mathfrak m(\partial B_a(S))}\geq \left(1-\epsilon\frac{\svolann{H}{a+r_0,b+r_0}}{\svolann{H}{d+r_0,b+r_0}}\right)\frac{\svolsp{H}{d+r_0}}{\svolsp{H}{a+r_0}}. \label{csph-com1}
\end{equation}
\end{Lem}

\begin{proof}
By $$\frac{\mathfrak m (A_{a,b}(S))}{\mathfrak m (S)}\geq (1-\epsilon)\cdot \frac{\svolann{H}{a+r_0,b+r_0}}{\svolsp{H}{r_0}},$$
we have that,
$$\mathfrak m(A_{a,b}(S))\frac{\svolsp{H}{r_0}}{\mathfrak m(S)} - \svolann{H}{a+r_0,b+r_0} \geq - \epsilon \svolann{H}{a+r_0,b+r_0}.$$
Thus for all $a< d< b$,
\begin{eqnarray*}
1 - \epsilon\frac{\svolann{H}{a+r_0,b+r_0}}{\svolann{H}{d+r_0,b+r_0}} &\leq &1 + \frac{\mathfrak m(A_{a,b}(S))}{\svolann{H}{d+r_0,b+r_0}}\frac{\svolsp{H}{r_0}}{\mathfrak m(S)} - \frac{\svolann{H}{a+r_0,b+r_0}}{\svolann{H}{d+r_0,b+r_0}}\\
&= & \frac{\mathfrak m(A_{a,b}(S))}{\svolann{H}{d+r_0,b+r_0}}\frac{\svolball{H}{r_0}}{\mathfrak m(S)} - \frac{\svolann{H}{a+r_0,d+r_0}}{\svolann{H}{d+r_0,b+r_0}}.
\end{eqnarray*}
By relative volume comparison Lemma~\ref{rel-vol},
$$\svolann{H}{a+r_0,d+r_0} \geq \mathfrak m(A_{a,d}(S))\cdot \frac{\svolball{H}{r_0}}{\mathfrak m(S)}.$$
So by \eqref{ann-sph},
\begin{eqnarray*}
1 - \epsilon\frac{\svolann{H}{a+r_0,b+r_0}}{\svolann{H}{d+r_0,b+r_0}} &\leq & \frac{\mathfrak m(A_{a,b}(S))}{\svolann{H}{d+r_0,b+r_0}}\frac{\svolball{H}{r_0}}{\mathfrak m(S)} - \frac{\mathfrak m(A_{a,d}(S))}{\svolann{H}{d+r_0,b+r_0}} \frac{\svolball{H}{r_0}}{\mathfrak m(S)}\\
&= & \frac{\mathfrak m(A_{d,b}(S))}{\svolann{H}{d+r_0,b+r_0}}\frac{\svolsp{H}{r_0}}{\mathfrak m(S)}\\
&\leq & \frac{\mathfrak m(\partial B_d(S))}{\svolsp{H}{d+r_0}}\frac{\svolsp{H}{r_0}}{\mathfrak m(S)}.
\end{eqnarray*}
Note that by \eqref{sph-com0} and \eqref{sph-com}, \eqref{lvol-com} implies that for $0\leq a<b$,
\begin{equation*}
\mathfrak m(A_{a, b}(S))\geq (1-\epsilon)(b-a)\mathfrak m(\partial B_a(S)), \text{ for } m=0 \text{ and } K=0;
\end{equation*}
\begin{equation*}
\mathfrak m(A_{a, b}(S))\geq (1-\epsilon)\frac{\svolann{H}{a+r_0, b+r_0}}{\svolsp{H}{a+r_0}}\mathfrak m(\partial B_a(S)), \text{ for } m\neq 0 \text{ or } K\neq 0.
\end{equation*}
Then same argument as above gives the results.
\end{proof}

Now we will use above volume estimates Lemma~\ref{rel-vol}, Lemma~\ref{qua-com} and Laplacian comparison Lemma~\ref{glap-com} to prove the Laplacian estimates Theorem~\ref{lap-main}.
\begin{proof}[Proof of Theorem~\ref{lap-main}]

By Lemma~\ref{glap-com}, we have that 
for $m=0, K=0$, $$\Delta d_{s, i}\leq 0;$$
for $m\neq 0$ or $K\neq 0$, 
$$\Delta d_{s, i} \leq (N-1)\frac{\op{sn}'_H(r+r_0)}{\op{sn}_H(r+r_0)}.$$
And for $a<d<b$ (see also \cite[Proposition 3.6]{CDNPSW})
\begin{eqnarray*}
-\kern-1em\int_{A_{a, d}(S_i)}\Delta d_{s, i} &=&\frac1{\mathfrak m_i(A_{a, d}(S_i))}\int_{\mathfrak Q(A_{a, d}(S_i)\cap \mathcal T^*_{d_{s_i}})}\int_{A_{a, d}(S_i)\cap X_{\alpha}}\Delta d_{s, i}h_{\alpha}(r)dr d\mathfrak q(\alpha)\\
&=&\frac1{\mathfrak m_i(A_{a, d}(S_i))} \int_{\mathfrak Q(A_{a, d}(S_i)\cap \mathcal T^*_{d_{s_i}})}\int_a^d\frac{h'_{\alpha}(r)}{h_{\alpha}(r)}h_{\alpha}(r)dr d\mathfrak q(\alpha)\\
&=&\frac1{\mathfrak m_i(A_{a, d}(S_i))} \int_{\mathfrak Q(A_{a, d}(S_i)\cap \mathcal T^*_{d_{s_i}})}h_{\alpha}(d)-h_{\alpha}(a) d\mathfrak q(\alpha)\\
& =& \frac{\mathfrak m_i(\partial B_d(S_i))-\mathfrak m_i(\partial B_a(S_i))}{\mathfrak m_i(A_{a, d}(S_i))}.
\end{eqnarray*}

Thus for $m=0$ and $K=0$, by \eqref{csph-com0} and \eqref{sph-com0}
$$-\kern-1em\int_{A_{a, d}(S_i)}\Delta d_{s, i} \geq 
\frac{-\epsilon_i(b-a)}{b-d-\epsilon_i(b-a)}\frac{\mathfrak m_i(\partial B_d(S_i))}{\mathfrak m_i(A_{a, d}(S_i))}\geq \frac{-\epsilon_i(b-a)}{b-d-\epsilon_i(b-a)}\frac1{d-a};$$
And for  $m\neq 0$ or $K\neq 0$, if $\svolsp{H}{d+r_0}>\svolsp{H}{a+r_0}$ which is always holds for $K\leq 0$, by \eqref{ann-sph} and \eqref{csph-com1}
\begin{eqnarray*}-\kern-1em\int_{A_{a, d}(S_i)}\Delta d_{s, i} &\geq & \frac{1}{\svolsp{H}{a+r_0}}\left((1-\epsilon_i C(N, H, a, b, d, r_0))\svolsp{H}{d+r_0}-\svolsp{H}{a+r_0}\right)\frac{\mathfrak m(\partial B_a(S))}{\mathfrak m(A_{a, d}(S))}\\
&\geq & \frac{(1-\epsilon_i C(N, H, a, b, d, r_0))\svolsp{H}{d+r_0}-\svolsp{H}{a+r_0}}{\svolann{H}{a+r_0, d+r_0}};
\end{eqnarray*}
If $\svolsp{H}{b+r_0}\leq \svolsp{H}{a+r_0}$ for $K>0$, as the $m=0$, $K=0$ case, 
\begin{eqnarray*}-\kern-1em\int_{A_{a, d}(S_i)}\Delta d_{s, i} &\geq & \frac{1}{\svolsp{H}{a+r_0}}\left(\svolsp{H}{d+r_0}-(1+\epsilon_i C(N, H, a, b, d, r_0))\svolsp{H}{a+r_0}\right)\frac{\mathfrak m(\partial B_d(S))}{\mathfrak m(A_{a, d}(S))}\\
&\geq & \frac{\svolsp{H}{d+r_0}-(1+\epsilon_i C(N, H, a, b, d, r_0))\svolsp{H}{a+r_0}}{\svolann{H}{a+r_0, d+r_0}};
\end{eqnarray*}
And by Lemma~\ref{ele-com0},
\begin{eqnarray*}
& & -\kern-1em\int_{A_{a, d}(S_i)} (N-1)\frac{\op{sn}'_H(r+r_0)}{\op{sn}_H(r+r_0)}\\
& =& \frac{1}{\mathfrak m(A_{a, d}(S_i))}\int_{A_{a, d}(S_i)\cap \mathcal T_{d_{s, i}}\setminus B_{\op{out}}}(N-1)\frac{\op{sn}'_{H}(r+r_0)}{\op{sn}_H(r+r_0)}d\mathfrak m\\
& \leq & \frac{1}{\mathfrak m(A_{a, b}(S_i))}\int_{\mathfrak Q(A_{a, d}(S_i)\cap \mathcal T^*_{d_{s, i}})}\int_a^d(N-1)\frac{\op{sn}'_{H}(r+r_0)}{\op{sn}_H(r+r_0)} \frac{h_{\alpha}(r)}{h_{\alpha}(a)}h_{\alpha}(a)dr d\mathfrak q(\alpha)\\
& \leq & \frac{1}{\mathfrak m(A_{a, b}(S_i))}\int_{\mathfrak Q(A_{a, d}(S_i)\cap \mathcal T^*_{d_{s, i}})}\int_a^d(N-1)\frac{\op{sn}'_{H}(r+r_0)}{\op{sn}_H(r+r_0)} \frac{\op{sn}^{N-1}_H(r+r_0)}{\op{sn}^{N-1}_H(a+r_0)}h_{\alpha}(a)dr d\mathfrak q(\alpha)\\
& = & \frac{1}{\mathfrak m(A_{a, b}(S_i))\op{sn}_H^{N-1}(a+r_0)}\int_{\mathfrak Q(A_{a, d}(S_i)\cap \mathcal T^*_{d_{s, i}})}\int_a^d d\op{sn}_H^{N-1}(r+r_0) h_{\alpha}(a) d\mathfrak q(\alpha)\\
&=&  \frac{\mathfrak m(\partial B_a(S_i))\left(\op{sn}_H^{N-1}(d+r_0)-\op{sn}_H^{N-1}(a+r_0)\right)}{\mathfrak m(A_{a, b}(S_i))\op{sn}_H^{N-1}(a+r_0)}\\
&\leq & \begin{cases}(1+\epsilon_i)\frac{\svolsp{H}{a+r_0}}{\svolann{H}{a+r_0, d+r_0}} \frac{\op{sn}_H^{N-1}(d+r_0)-\op{sn}_H^{N-1}(a+r_0)}{\op{sn}_H^{N-1}(a+r_0)}, & \op{sn}_H^{N-1}(d+r_0)>\op{sn}_H^{N-1}(a+r_0)\\
\frac{\svolsp{H}{a+r_0}}{\svolann{H}{a+r_0, d+r_0}} \frac{\op{sn}_H^{N-1}(d+r_0)-\op{sn}_H^{N-1}(a+r_0)}{\op{sn}_H^{N-1}(a+r_0)}, & \op{sn}_H^{N-1}(d+r_0)\leq \op{sn}_H^{N-1}(a+r_0)\end{cases}\\
&\leq & (1+\epsilon_i) \frac{\svolsp{H}{d+r_0}-\svolsp{H}{a+r_0}}{\svolann{H}{a+r_0, d+r_0}}
\end{eqnarray*}

We have that
\begin{equation}
-\kern-1em\int_{A_{a, d}(S_i)}\Delta d_{s, i}\geq (1-\Psi(\epsilon_i | N, K, a, d, b, m))-\kern-1em\int_{A_{a, d}(S_i)} (N-1)\frac{\op{sn}'_H(r+r_0)}{\op{sn}_H(r+r_0)}.
\end{equation}

Now as the proof of Claim 1 in \cite[Proposition 4.3]{Ch}, passing to the limit, we have that on $A_{a, d}(S)$

(i) for $m=0$ and $K=0$, 
$$\Delta d_s=0;$$

(ii) for $m\neq 0$ or $K\neq 0$, 
$$\Delta d_s=(N-1)\frac{\op{sn}'_H(d_s+r_0)}{\op{sn}_H(d_s+r_0)}.$$
\end{proof}

\begin{Cor} Let the assumption be as in Theorem~\ref{lap-main}. For $m\neq 0$ or $K\neq 0$, let
$$f_H(x)=f_H(d_s(x)), \quad f_H(r)=\int \op{sn}_H(r+r_0)dr.$$
Then in $A_{a, b}(S)$
$$\nabla f_H=\op{sn}_{H}(d_s+r_0)\nabla d_s,$$
and
$$\Delta f_H= N\op{sn}'_{H}(d_s+r_0).$$
\end{Cor}

\section{Hessian estimates}

In this section we will use the Laplacian estimates Theorem~\ref{lap-main} to give an estimate of $\op{Hess}(d_s)$ for $m=0$ and $K=0$ and estimates of $\op{Hess}(f_H)$ for $m\neq 0$ or $K\neq 0$. 

\begin{Thm}[Hessian estimates] \label{hess-est}
Let the assumption be as in Theorem~\ref{lap-main} and let $a<d<b$.

(i) For $m=0$, $K=0$, in $A_{a+\frac{d-a}4, b-\frac{d-a}4}(S)$ $\mathfrak m$-a.e.
\begin{equation}
\op{Hess}(d_s)=0.
\label{hess-0}
\end{equation}

(ii) For $m\neq 0$ or $K\neq 0$, in $A_{a+\frac{d-a}4, b-\frac{d-a}4}(S)$ $\mathfrak m$-a.e.
\begin{equation}
\op{Hess}(f_H)=\op{sn}'_H(d_s+r_0).\label{hess-1}
\end{equation}

\end{Thm}
\begin{proof}

Take a cut-off function $\phi: X\to [0, 1]$ as in Lemma~\ref{cut-off} such that 

(1) $\phi=1$ on $A_{a+\frac{d-a}4, b-\frac{d-a}4}(S)$, $\phi=0$ on $X\setminus  A_{a+\frac{d-a}5, b-\frac{d-a}5}(S)$;

(2) $|\Delta \phi|+|\nabla \phi|\leq C(K, N, a, b)$ a.e. on $A_{a+\frac{d-a}5, b-\frac{d-a}5}(S)$.

For $m=0, K=0$, by the improved Bochner inequality Theorem~\ref{Boc-ine} and Theorem~\ref{lap-main},
\begin{eqnarray*}
0&=&-\frac12\int_X \Gamma(\phi, |\nabla d_s|^2)=\int_X\phi \frac12 \Delta|\nabla d_s|^2\\
&\geq & \int_X\phi\left( |\op{Hess}(d_s)|^2_{\op{HS}}+ \Gamma(d_s, \Delta d_s)\right)=\int_X\phi |\op{Hess}(d_s)|^2_{\op{HS}}.
\end{eqnarray*}
And thus in $A_{a+\frac{d-a}4, b-\frac{d-a}4}(S)$ $\mathfrak m$-a.e.
\begin{equation*}\op{Hess}(d_s)=0. \end{equation*}

%Let $\delta$ goes to $0$. And by using \eqref{hess-1}, we have  that  for a.e $y\in A$ and $t$,
%$$d^2(x, F_t(y))-t^2-d^2(x, y)=0.$$

For $m\neq 0, K=0$, $f_0=\frac12 (r+r_0)^2$,
\begin{eqnarray*}
0&=&-\int_X \Gamma(\phi, \frac12|\nabla f_0|^2-f_0)=\int_X\phi \left(\frac12 \Delta|\nabla f_0|^2-\Delta f_0\right)\\
&\geq & \int_X\phi\left( |\op{Hess}(f_0)|^2_{\op{HS}}+ \Gamma(f_0, \Delta f_0)-N\right)\geq \int_X\phi |\op{Hess}(f_0)-1|^2_{\op{HS}}.
\end{eqnarray*}

For $K=N-1$, $f_1=-\cos (r+r_0)$,
\begin{eqnarray*}
0&=&-\frac12\int_X \Gamma(\phi, |\nabla f_1|^2+f_1^2-1)=\frac12\int_X\phi \left(\Delta|\nabla f_1|^2+2|\nabla f_1|^2+2f_1\Delta f_1\right)\\
&\geq & \int_X\phi\left( |\op{Hess}(f_1)|^2_{\op{HS}}+(N-1)|\nabla f_1|^2+ \Gamma(f_1, \Delta f_1)+|\nabla f_1|^2+f_1\Delta f_1\right)\\
&=& \int_X\phi\left( |\op{Hess}(f_1)|^2_{\op{HS}}+ (N-1)\sin^2(r+r_0)-N\sin^2(r+r_0) +\sin^2(r+r_0)-N\cos^2(r+r_0)\right)\\
&\geq& \int_X\phi |\op{Hess}(f_1)-\cos(r+r_0)|^2_{\op{HS}}.
\end{eqnarray*}

For $K=-(N-1)$, $f_{-1}=\cosh (r+r_0)$,
\begin{eqnarray*}
0&=&-\frac12\int_X \Gamma(\phi, |\nabla f_{-1}|^2-f_{-1}^2+1)=\frac12\int_X\phi \left(\Delta|\nabla f_{-1}|^2-2|\nabla f_{-1}|^2-2f_{-1}\Delta f_{-1}\right)\\
&\geq & \int_X\phi\left( |\op{Hess}(f_{-1})|^2_{\op{HS}}-(N-1)|\nabla f_{-1}|^2+ \Gamma(f_{-1}, \Delta f_{-1})-|\nabla f_{-1}|^2-f_{-1}\Delta f_{-1}\right)\\
&=& \int_X\phi\left( |\op{Hess}(f_{-1})|^2_{\op{HS}}- (N-1)\sinh^2(r+r_0)+N\sinh^2(r+r_0) -\sinh^2(r+r_0)-N\cosh^2(r+r_0)\right)\\
&\geq& \int_X\phi |\op{Hess}(f_{-1})-\cosh(r+r_0)|^2_{\op{HS}}.
\end{eqnarray*}
\end{proof}

%new section here
\section{Pythagoras theorem and Cosine law}

Let $(X, d, \mathfrak m)$ be as in Theorem~\ref{lap-main}. In this section, we will use a method as in \cite{CC1} to show that the metric in $A_{a+\delta, b-\delta}(S)$ (some $\delta>0$) satisfies  Pythagoras theorem for $m=0$ and $K=0$ or Cosine law for $m\neq 0$ or $K\neq 0$. 

%\subsection{Pythagoras theorem and Cosine law}

%In this section, we will see that the metric on $A_{a+\frac{d-a}3, b-\frac{d-a}3}(S)$ satisfies the Pythagoras theorem for $m=0$, $K=0$ and Cosine law otherwise.

First note that for $t\in [-(b-a)/3, (b-a)/3]$, $td_s$ is $c$-concave, thus we can apply the differential formula in Section 2.5 to $d_s$.
\begin{Lem} \label{c-concave}
Let the assumption be as in Theorem~\ref{lap-main}. Then $td_s$ is $c$-concave for $t\in [-\frac{b-a}3, \frac{b-a}3]$, and 
$$(td_s)^c(y)=-td_s-\frac{t^2}2.$$
\end{Lem}
\begin{proof}
The proof is the same as in \cite{CC1}. First take small $\tau>0$ and consider
 $$\mathcal T^i_{a, b}=\{\mathcal T_{d_{s, i}}\cap A_{a, b-\tau}(S_i), \,  \mathcal T_{d_{s, i}}\cap A_{b-\tau, b}(S_i)\neq \emptyset \}.$$
Then for any $x\in \mathcal T^i_{a, b}$, there is $y\in A_{b-\tau, b}(S_i)$, such that
$$d_{s, i}(y)-d_{s, i}(x)=d_i(x, y).$$
Since $\mathfrak m_i(X_i\setminus \mathcal T_{d_{s, i}})=0$, 
$$\mathfrak m_i(\mathcal T_{d_{s, i}}\cap A_{b-\tau, b}(S_i))=\mathfrak m_i(A_{b-\tau, b}(S_i)).$$
And by volume element comparison Lemma~\ref{vol-ele-com} and \eqref{lqvol-com}, 
\begin{eqnarray*}
\frac{\mathfrak m(\mathcal T^i_{a, b})}{\mathfrak m_i(A_{a, b}(S_i))}&=& \frac{\mathfrak m_i(\mathcal T^i_{a, b})}{\svolann{H}{a+r_0, b-\tau+r_0}}\frac{\svolann{H}{a+r_0, b+r_0}}{\mathfrak m_i(A_{a, b}(S_i))}\frac{\svolann{H}{a+r_0, b-\tau+r_0}}{\svolann{H}{a+r_0, b+r_0}}\\
&\geq & \frac{A_{b-\tau, b}(S_i)}{\svolann{H}{b-\tau, b}}\frac{\svolann{H}{a+r_0, b+r_0}}{\mathfrak m_i(A_{a, b}(S_i))}\frac{\svolann{H}{a+r_0, b-\tau+r_0}}{\svolann{H}{a+r_0, b+r_0}}\\
&\geq & (1-\epsilon_i)\frac{\mathfrak m_i(\partial B_a(S_i))}{\svolsp{H}{a+r_0}}\frac{\svolann{H}{a+r_0, b+r_0}}{\mathfrak m_i(A_{a, b}(S_i))}\frac{\svolann{H}{a+r_0, b-\tau+r_0}}{\svolann{H}{a+r_0, b+r_0}}\\
&\geq & (1-\epsilon_i)\frac{\svolann{H}{a+r_0, b-\tau+r_0}}{\svolann{H}{a+r_0, b+r_0}}.
\end{eqnarray*}
Since $\tau$ can be taken arbitrary small, without loss of generality, we may let $\tau\to 0$ and derive that
 for each $x\in A_{a, b}(S_i)$, there is $y\in \partial B_b(S_i)$ such that 
\begin{equation}d_i(x, y)\geq b-d_{s, i}(x)\geq (1-\epsilon_i)d_i(x, y). \label{geo-con}\end{equation}

Since $d_s$ is $1$-Lipschitz, as in \cite{Gi13}, for $x, y\in X$,
$$td_s(x)-td_s(y)\leq |t|d(x, y)\leq \frac{t^2}{2}+\frac{d^2(x, y)}{2}.$$
And by the definition of $c$-transform,
$$(td_s)^c(y)=\inf_{x\in X}\frac{d^2(x,y)}{2}-td_s(x)\geq -td_s(y)-\frac{t^2}{2}.$$

For the opposite inequality, note that for $y\in A_{a, b}(S)$, there is $y_i\in A_{a, b}(S_i)$ and $y_i^-\in \partial B_a(S_i)$, $y_i^+\in \partial B_b(S_i)$ such that $y_i\to y$
$$d_i(y_i, y_i^-)=d_{s, i}(y_i)-a, \, d_{s, i}(y_i^+)-d_{s, i}(y_i)-d_i(y_i^+, y_i)\geq -\epsilon_i.$$
Then
$$b-a\leq d_i(y_i^+, y_i^-)\leq d_i(y_i, y_i^-)+d_i(y_i, y_i^+)\leq b-a+\epsilon_i.$$

Let $\gamma_i^{-}$ be a unit speed minimal geodesic from $y_i^-$ to $y$ and let $\gamma_i^+$ be a unit speed minimal geodesic from $y$ to $y_i^+$, $\gamma_i\to \gamma$. Assume $\gamma_i^-\cup\gamma_i^+\to \gamma\in X$.

For $y\in A_{\frac{b+2a}3, \frac{2b+a}3}(S)$,  for each $t\in (-\frac{b-a}3, \frac{b-a}3)$, we can take $\gamma_t=\gamma(d_s(y)+t)$ such that
$$d_s(\gamma_t)-d_s(y)=t=\op{sign}(t)d(y, \gamma_t).$$
Thus
$$\left(td_s\right)^c(y)\leq \frac{d^2(\gamma_t, y)}2-td_s(\gamma_t)=-td_s(y)-\frac{t^2}2.$$
\end{proof}

Take a cut-off function $\phi: X\to [0, 1]$ as in the proof of Theorem~\ref{hess-est}. Consider the vector field $\phi\nabla d_s$. Then as the discussion in section 3.4 of \cite{CDNPSW}, by \cite{AT14} (see Theorem~\ref{exist-RLF}), we know that the Regular Lagrangian flow $F_t$ for $\phi\nabla d_s$ exists and is unique. 

Using Hessian estimates Theorem~\ref{hess-est}, Lemma~\ref{c-concave} and the differential formula Theorem~\ref{first-diff1}, Corollary~\ref{first-diff2} and Corollary~\ref{second-diff2}, we will derive Pythagoras theorem for $m=0$, $K=0$,  and Cosine law for $m\neq 0$ or $K\neq 0$. 

Note that by Corollary~\ref{first-diff2}, we have that for a Lipschitz map $f$,
\begin{equation}f(p)= f(x)-\lim_{\delta\to 0}\int_{\delta}^{1} d(x, p)\left<\nabla f, \nabla d_p\right>(\gamma_{p, x}(t)) dt, \label{diff-form-1}\end{equation}
where $\gamma_{p, x}$ is a constant geodesic from $p$ to $x$ and $d_p$ is the distance function from $p$. In the following for simplicity we will write \eqref{diff-form-1} just as 
$$f(x)-f(p)=\int_0^1d(x, p)\left<\nabla f, \nabla d_p\right>(\gamma_{p, x}(t)) dt.$$

\begin{Thm}[Pythagoras theorem and Cosine law] \label{cos-law}
Let the assumption be as in Theorem~\ref{lap-main}. Let $F_t$ be the regular Lagrangian flow of $\phi\nabla d_s$. Denote $A=A_{a+\frac{b-a}3, b-\frac{b-a}3}(S)$.  Let $B=B_r(x)\subset A$ be a geodesic ball in $A$.

(i) For $m=0$ and $K=0$, 
\begin{equation}
\int_{B\times B}\left|d^2(x, y)-(d_s(x)-d_s(y))^2-d^2\left(x, F_{-d_s(y)+d_s(x)}(y)\right)\right| d\mathfrak m(x) \mathfrak m(y)=0;
\end{equation}

(ii) For $m\neq 0$ or $K\neq 0$,

(ii-1) if $K=0$, we have that
\begin{equation}
\int_{B\times B}\left|\frac{(d_s(x)+r_0)^2+(d_s(y)+r_0)^2-d^2(x, y)}{2(d_s(x)+r_0)(d_s(y)+r_0)}-\frac{2\left(r_0+d_s(x)\right)^2-d^2\left(x, F_{-d_s(y)+d_s(x)}(y)\right)}{2\left(r_0+d_s(x)\right)^2}\right| d\mathfrak m(x) \mathfrak m(y)=0;
\end{equation}

(ii-2) if $K=\pm(N-1)$, and thus $H=\pm 1$,
\begin{equation}
\int_{B\times B}\left|\frac{\op{sn}'_H(d(x,y))-\op{sn}'_H(d_s(x)+r_0)\op{sn}'_H(d_s(y)+r_0)}{\op{sn}_H(d_s(x)+r_0)\op{sn}_H(d_s(y)+r_0)}-\frac{\op{sn}'_H\left(d\left(x, F_{-d_s(y)+d_s(x)}(y)\right)\right)-{\op{sn}'}^2_H\left(r_0+d_s(x)\right)}{\op{sn}^2_H\left(r_0+d_s(x)\right)}\right| =0;
\end{equation}

\end{Thm}

\begin{proof}
Assume for $x, y\in B$, $d_s(x)=t_0$, $d_s(y)=t$. Let $\gamma_{\tau}(l): [0, 1]\to B$ be a constant speed geodesic from $x$ to $F_{-\tau}(y)$, $\tau\in[0, t-t_0]$.

For $m=0$, $K=0$, 
\begin{eqnarray*}
& &\int_{B\times B}\frac12\left|d^2(x, y)-(t-t_0)^2-d^2\left(x, F_{t_0-t}(y)\right)\right| d\mathfrak m(x) d\mathfrak m(y)\\
&=& \int_{B\times B}\frac12\left|\left.\left(d^2(x, F_{-\tau}(y))-(t-t_0-\tau)^2\right)\right|_{t-t_0}^0\right| d\mathfrak m(x) d\mathfrak m(y)\\
&\leq & \int_{B\times B} \int_0^{|t-t_0|}\left|d(x, F_{-\tau}(y))\left<\nabla d_x, \nabla d_s\right>(F_{-\tau}(y))-(t-t_0-\tau) \right|d\tau d\mathfrak m(x)d\mathfrak m(y) \\
&=& \int_{B\times B} \int_0^{|t-t_0|}\left|d(x, F_{-\tau}(y))\left<\nabla d_x, \nabla d_s\right>(F_{-\tau}(y))-\left.d_s(\gamma_{\tau}(l))\right|^{1}_0 \right|d\tau  d\mathfrak m(x)d\mathfrak m(y) \\
& \leq & \int_{B\times B} \int_0^{|t-t_0|}\int_0^{1}\left|d(x, F_{-\tau}(y))\left(\left<\nabla d_x, \nabla d_s\right>(F_{-\tau}(y))-\left<\nabla d_x, \nabla d_s\right>(\gamma_{\tau}(l))\right)\right| dl d\tau d\mathfrak m(x)d\mathfrak m(y) \\
&\leq & \int_0^1 \int_{l}^{1}\int_{B\times B}\int_0^{|t-t_0|}d^2(x, F_{-\tau}(y))\left|\op{Hess}_{d_s}(\gamma'_{\tau}(\xi), \gamma'_{\tau}(\xi))\right| d\tau d\mathfrak m(x)d\mathfrak m(y)  d\xi dl\\
&\leq & \int_0^{\frac{b-a}3}\int_0^1 \int_{l}^{1}\int_{B\times B}d^2(x, F_{-\tau}(y))\left|\op{Hess}_{d_s}(\gamma'_{\tau}(\xi), \gamma'_{\tau}(\xi))\right| d\mathfrak m(x)d\mathfrak m(y)  d\xi dld\tau\\
& =& 0.
\end{eqnarray*}

Now prove the cosine law for $K=0$ and $m\neq 0$ where $f_0(x)=\frac12 (d_s+r_0)^2$.
\begin{eqnarray*}
& & \int_{B\times B}\left|\frac{(t_0+r_0)^2+(t+r_0)^2-d^2(x, y)}{2(t_0+r_0)(t+r_0)}-\frac{2\left(r_0+t_0\right)^2-d^2\left(x, F_{t_0-t}(y)\right)}{2\left(r_0+t_0\right)^2}\right| d\mathfrak m(x) \mathfrak m(y)\\
&=& \int_{B\times B}\left|\left.\frac{(t_0+r_0)^2+(t+r_0-\tau)^2-d^2(x, F_{-\tau}(y))}{2(t_0+r_0)(t+r_0-\tau)}\right|_{t-t_0}^0\right| d\mathfrak m(x) \mathfrak m(y)\\
&\leq & \frac{1}{2(a+r_0)^3}\int_{B\times B}\int_0^{|t-t_0|} \left|(r_0+t-\tau)^2-(t_0+r_0)^2 + d^2(x, F_{-\tau}(y))\right.\\
& & \left.-d(x, F_{-\tau}(y))\left<\nabla d_x, \nabla (d_s+r_0)^2\right>(F_{-\tau}(y))\right| d\tau\\
& = & \frac{1}{(a+r_0)^3}\int_{B\times B}\int_0^{|t-t_0|} \left|\left.f_0(\gamma_{\tau}(l)\right|_0^1 + \frac12d^2(x, F_{-\tau}(y))-d(x, F_{-\tau}(y))\left<\nabla d_x, \nabla f_0\right>(F_{-\tau}(y))\right| d\tau\\
& = & \frac{1}{(a+r_0)^3}\int_{B\times B}\int_0^{|t-t_0|} \left|\left.\left(f_0(\gamma_{\tau}(l)) - \frac12\left(d(x, F_{-\tau}(y))-d(x, \gamma_{\tau}(l))\right)^2\right)\right|_{0}^1\right.\\
&  & \left.-d(x, F_{-\tau}(y))\left<\nabla d_x, \nabla f_0\right>(F_{-\tau}(y))\right| d\tau\\
&\leq & \frac{1}{(a+r_0)^3}\int_{B\times B} \int_0^{|t-t_0|}\int_{0}^1d(x, F_{-\tau}(y))\left|\left<\nabla f_0, \nabla d_x\right>(\gamma_{\tau}(l))  -\left<\nabla d_x, \nabla f_0\right>(F_{-\tau}(y))\right.\\
& &\left.+ d(x, F_{-\tau}(y))-d(x, \gamma_{\tau}(l))\right| dld\tau \\
& \leq & \frac{1}{(a+r_0)^3}\int_0^1 \int_{l}^{1}\int_{B\times B}\int_0^{|t-t_0|} d^2(x, F_t(y))\left|\op{Hess}_f(\nabla d_x, \nabla d_x)(\gamma_{\tau}(\xi))-1\right| d\tau d\mathfrak m(x)d\mathfrak m(y)d\xi dl\\
& =& 0
\end{eqnarray*}

For $K=(N-1)$, consider $f_1=-\cos(d_s+r_0)$,  we can see that %let $\tilde \gamma_{\tau}$ be a unit speed representation of $\gamma_{\tau}$.

\begin{eqnarray*}
& & \int_{B\times B}\left|\frac{\cos(d(x, y))-\cos(t_0+r_0)\cos(t+r_0)}{\sin(t_0+r_0)\sin(t+r_0)}-\frac{\cos\left(d\left(x, F_{t_0-t}(y)\right)\right)-\cos^2\left(r_0+t_0\right)}{\sin^2\left(r_0+t_0\right)}\right|d\mathfrak m(x) d\mathfrak m(y) \\
&=& \int_{B\times B}\left|\left.\frac{\cos d(x, F_{-\tau}(y))-\cos(t_0+r_0)\cos(t+r_0-\tau)}{\sin(t_0+r_0)\sin(t+r_0-\tau)}\right|_{0}^{t-t_0}\right|d\mathfrak m(x) d\mathfrak m(y) \\
&\leq & c \int_{B\times B}\int_0^{|t-t_0|}  \left|\sin d(x, F_{-\tau}(y))\left<\nabla d_x, \nabla f_1\right>(F_{-\tau}(y))- \cos (t_0+r_0)+\cos d(x, F_{-\tau}(y))\cos(t+r_0-\tau)\right| d\tau \\
&= & c \int_{B\times B}\int_0^{|t-t_0|}  \left|\sin d(x, F_{-\tau}(y))\left<\nabla d_x, \nabla f_1\right>(F_{-\tau}(y))+ \cos d(x, \gamma_{\tau}(0))f_1(\gamma_{\tau}(0))-\cos d(x, \gamma_{\tau}(1))f_1(\gamma_{\tau}(1))\right| d\tau \\
&= & c \int_{B\times B}\int_0^{|t-t_0|}  \left|\left.\left(\sin d(x, \gamma_{\tau}(l))\left<\nabla d_x, \nabla f_1\right>(F_{-\tau(y)})- \cos d(x, \gamma_{\tau}(l))f_1(\gamma_{\tau}(l))\right)\right|_0^1\right| d\tau \\
&\leq & c \int_0^{\frac{b-a}{3}}\int_{B\times B}\left|\int_0^1d(x, F_{-\tau}(y))\cos \left(d(x, \gamma_{\tau}(l)) \left<\nabla f_1, \nabla d_x\right>(F_{-\tau}(y))\right.\right.\\
& & \left.\left.-\cos d(x, \gamma_{\tau}(l)) \left<\nabla f_1, \nabla d_x\right>(\gamma_{\tau}(l))+\sin d(x, \gamma_{\tau}(l)) f_1(\gamma_{\tau}(l))\right)dl\right| d\mathfrak m(x)d\mathfrak m(y)d\tau\\
&= & c \int_0^{\frac{b-a}{3}}\int_{B\times B}\left|\int_0^1d(x, F_{-\tau}(y))\left(\cos d(x, \gamma_{\tau}(l))\left.\left<\nabla f, \nabla d_x\right>(\gamma_{\tau}(\xi))\right|_l^1+\left.\sin d(x, \gamma_{\tau}(\xi))\right|_0^l f_1(\gamma_{\tau}(l))\right)dl\right| \\
&= & c \int_0^{\frac{b-a}{3}}\int_{B\times B}\left|\int_0^1d(x, F_{-\tau}(y))\cos d(x, \gamma_{\tau}(l))\left.\left<\nabla f, \nabla d_x\right>(\gamma_{\tau}(\xi))\right|_l^1 dl\right.\\
& & +\left.\int_0^1\int_0^l \cos d(x, \gamma_{\tau}(\xi)) d^2(x, F_{-\tau}(y))f_1(\gamma_{\tau}(l))d\xi dl \right|\\
&= & c \int_0^{\frac{b-a}{3}}\int_{B\times B}\left|\int_0^1d^2(x, F_{-\tau}(y))\cos d(x, \gamma_{\tau}(l))\int_l^1\op{Hess}_{f_1}(\gamma'_{\tau}(\xi), \gamma'_{\tau}(\xi))d\xi dl\right.\\
& & +\left.\int_0^1\int_l^1 \cos d(x, \gamma_{\tau}(l)) d^2(x, F_{-\tau}(y))f_1(\gamma_{\tau}(\xi))d\xi dl \right|\\
&= & c \int_0^{\frac{b-a}{3}}\int_0^1 \int_{l}^{1}\int_{B\times B} d^2(x, F_{-\tau}(y))\left|\cos d(x, \gamma_{\tau}(l))\right| \left|\op{Hess}_{f_1}(\nabla d_x, \nabla d_x)(\gamma_{\tau}(\xi))+ f_1(\gamma_{\tau}(\xi))\right| \\
& = & 0.
\end{eqnarray*}

For $K=-(N-1)$, the same argument as $K=(N-1)$ gives the result.
\end{proof}
By the Pythagoras theorem and Cosine law above, we have that
\begin{Thm} \label{flow-con}
The regular Lagrangian flow $F: (-\frac{b-a}3, \frac{b-a}3)\times A\to A_{a, b}(S)$ admits a continuous representation with respect to the measure $\mathcal L^1\times \mathfrak m$ which we still denote by $F$. In particular the Pythagoras theorem or Cosine law in Theorem~\ref{cos-law} holds pointwise for $F$.
\end{Thm}
\begin{proof}
By the definition of regular Lagrangian flow, for $\mathfrak m$-a.e. $x$, $t\mapsto F_t(x)$ is continuous. And since
$$d(F_t(x), F_l(y))= d(F_t(x), F_l(y))-d(F_t(x), F_l(x))+d(F_t(x), F_l(x))\leq d(F_l(y), F_l(x))+d(F_t(x), F_l(x)),$$
it sufficient to show that for $\mathfrak m\times \mathfrak m$-a.e. $x, y\in A$ with $d_s(x)=d_s(y)=l$ and $d(x, y)<\delta<1$, for $t\in (-\frac{b-a}3, \frac{b-a}3)$,
\begin{equation} \label{con-equ}
d(F_t(x), F_t(y))\leq \Psi(\delta | K, a, b, r_0).
\end{equation}

If $m=0$, $K=0$, by Pythagoras theorem, we see that for $\mathfrak m\times \mathfrak m$-a.e. $x, y\in A$ with $d_s(x)=d_s(y)=l$, for $t\in (-\frac{b-a}3, \frac{b-a}3)$,
$$d(x, y)=d(F_t(x), F_t(y));$$

If $K=0$, $m\neq 0$, by Cosine law,
$$\frac{2(l+r_0)^2-d^2(x, y)}{2(l+r_0)^2}=\frac{2(l+r_0+t)^2-d^2(F_t(x), F_t(y))}{2(l+r_0+t)^2},$$
that is 
$$d^2(F_t(x), F_t(y))=\left(\frac{l+t+r_0}{l+r_0}\right)^2d^2(x, y)\leq \frac{(b+r_0)^2}{(a+r_0)^2}d^2(x,y).$$

If $K=(N-1)$, by Cosine law,
$$\frac{\cos d(x, y)-\cos^2(l+r_0)}{\sin^2(l+r_0)}=\frac{\cos d(F_t(x), F_t(y))-\cos^2(l+t+r_0)}{\sin^2(l+t+r_0)},$$
that is 
$$\frac{\cos d(F_t(x), F_t(y))-1}{\sin^2(b+r_0)}\leq \frac{\cos d(F_t(x), F_t(y))-1}{\sin^2(l+t+r_0)}=\frac{\cos d(x, y)-1}{\sin^2(l+r_0)}\leq \Psi(\delta | a, b, r_0).$$
Thus \eqref{con-equ} holds. 

For $K=-(N-1)$, it is similar as the $K=(N-1)$ case.
\end{proof}

\section{Warped product structure} In this section, we will finish the proof of Theorem~\ref{main}. First by the Pythagoras theorem or Cosine law, we derive a warped product structure $(a', b')\times_{\op{sn}'_H(r)+\frac{m}{N-1}\op{sn}_H(r)}Y$ of $A_{a', b'}(S)$ and see that $Y$ has finite components. Then using this warped product and the $\RCD$-condition of $A_{a', b'}(S)$ we prove that each component of $Y$ is infinitesimally Hilbertian and satisfies the Sobolev to Lipschitz property (see Theorem~\ref{measure-pro}). And last, by a methods as in \cite{Ket1}, we show that each component of $Y$ is a $\RCD$-space (see Theorem~\ref{Y-rcd}).

\subsection{Warped product structure}

%Using the Pythagoras theorem and Cosine law in above section,  we have the following  metric warped product structure. 

Let $a'=a+(b-a)/3, b'=b-(b-a)/3$ and let $A=A_{a', b'}(S)$ be as in Theorem~\ref{cos-law}.  Let $F: [-\frac{b-a}3, \frac{b-a}3]\times A\to A_{a, b}(S)$ be a continuous map as in Theorem~\ref{flow-con}. 
Let $Y=\partial B_{a'}(S)$ and let $\iota: Y\to A_{a, b}(S)$ be the inclusion map. Assume $Y$ has one component.
For $y_1, y_2\in Y$, define $d_Y(y_1, y_2)$ and the measure $\mathfrak m_Y$ as follows:

(i) for $m=0$, $K=0$,
$$ d_Y(y_1, y_2)= d(y_1, y_2);$$
And for $\bar E\subset Y$, 
\begin{equation}\mathfrak m_Y (\bar E)=\frac{\mathfrak m(\{x\in A, F_{-d_s(x)+a'}(x)\in \bar E\})}{b'-a'};\label{measure-1}\end{equation}

(ii) for $m\neq 0$ or $K\neq 0$, %if $\frac{d^2(y_1, y_2)}{2(r_0+a+(b-a)/3)^2}\leq 2$
$$d_Y(y_1, y_2)=\frac1{\op{sn}_H(r_0+a')}\inf\{L(\bar \gamma),\, \bar \gamma\in \iota(Y), \bar\gamma_0=y_1, \bar \gamma_1=y_2\},$$
where
$$L(\bar\gamma)=\limsup_{\delta\to 0}\left\{\sum d(\bar\gamma(t_i), \bar\gamma(t_{i+1})),\, d(\bar\gamma(t_i), \bar\gamma(t_{i+1}))\leq \delta,  0\leq t_1\leq \cdots \leq t_n\leq b \text{ is a divison of }\bar\gamma\right\}$$
%i.e., the intrinsic metric on $Y$ is equal $\op{sn}_H(r+a')d_Y$ locally.
For $\bar E\in Y$, define the measure $\mathfrak m_Y(\bar E)$ as
\begin{equation}\mathfrak m_Y(\bar E)=\frac{\mathfrak m(\{x\in A, F_{-d_s(x)+a'}(x)\in \bar E\})}{\int_{a'}^{b'}\op{sn}^{N-1}_H(t+r_0)dt}.\label{measure-2}\end{equation}

Define a map from $A$ to a warped product space as the following:

For $m=0$ and $K=0$, define
$$\Phi: A\to (a', b')\times Y, \, x\mapsto \Phi(x)=(d_s(x), F_{-d_s(x)+a'}(x)).$$
Then by Pythagoras theorem, for $x, x'\in A$,
$$d^2(x, x')=(d_s(x)-d_s(x'))^2+ d^2(F_{-d_s(x)+a'}(x), F_{-d_s(x')+a'}(y))=d^2(\Phi(x), \Phi(x')).$$

For $m\neq 0$ or $K\neq 0$, define
$$\Phi: A\to (a'+r_0, b'+r_0)\times_{\op{sn}_H(r)} Y, \, x\mapsto \Phi(x)=(d_s(x)+r_0, F_{-d_s(x)+a'}(x)).$$
And by Cosine law and the definition fo $d_Y$, we have that for $x, x'\in A$
$$d(x, x')=d_w(\Phi(x), \Phi(x')).$$

In fact, assume $\gamma$ is a minimal geodesic from $x$ to $x'$ with $d(x, x')$ sufficient small such that $\gamma\subset A$. Let $F_{-d_s(\gamma(t))+a'} (\gamma(t))=\bar \gamma(t)$. Dived $\gamma$ with
 $0=t_0\leq t_1\leq \cdots \leq t_n=1$ such that $d(\gamma(t_i), \gamma(t_{i+1}))\leq \delta$. Then
\begin{eqnarray*}
d(x, x') &= &\lim_{\delta\to 0}\sum d(\gamma(t_i), \gamma(t_{i+1}))\\
&=& \lim_{\delta\to 0}\sum \frac{\op{sn}_H(d_s(\gamma(\xi_i))+r_0)}{\op{sn}_H(a'+r_0)}d(\bar \gamma(t_i), \bar \gamma(t_{i+1}))\\
&=& \lim_{\delta\to 0}\sum \op{sn}_H(d_s(\gamma(\xi_i))+r_0)d_Y(\bar \gamma(t_i), \bar \gamma(t_{i+1}))\\
&=& d_w(\Phi(x), \Phi(x')).
\end{eqnarray*}
 %by Cosine law we only need to see that for any $y_1, y_2\in Y$, $d(y_1, y_2)<\delta$ for some $\delta>0$,
%$$d(y_1, y_2)=d_w(\Phi(y_1), \Phi(y_2)).$$
And more precisely,  we have that for $d_Y(y_1, y_2)\leq \delta$, some $\delta>0$,

(ii-1) if $K=0$ and $m\neq 0$, 
$$ d_Y(y_1, y_2)=\arccos\left(1-\frac{d^2(y_1, y_2)}{2(r_0+a+(b-a)/3)^2}\right);$$

(ii-2) if $K=N-1$,
$$d_Y(y_1, y_2)=\arccos\left(\frac{\cos d(y_1, y_2)-1}{\sin^2(r_0+a+(b-a)/3)^2}+1\right);$$

(ii-3) if $K=-(N-1)$,
$$d_Y(y_1, y_2)=\arccos\left(\frac{1-\cosh d(y_1, y_2)}{\sinh^2(r_0+a+(b-a)/3)^2}+1\right).$$

If $Y$ has more than one components, by above discussion, restricted to each component of $A$, $\Phi$ is an isometry. Now as the discussion in \cite[Claim 5.8]{Hu}, by relative volume comparison, we know that the number of $Y$'s components $\leq C(N, H, D, b, a)$. In fact, for each component $Y_k$, there is $y_k\in Y_k$ such that $x_k=F_{(b'-a')/2}(y_k)\in A$ and thus $B_{(b'-a')/3}(x_k)\subset A$. And for $l\neq k$, $B_{(b'-a')/3}(x_k)\cap B_{(b'-a')/3}(x_l)=\emptyset$. By relative volume comparison and that $A\subset B_{D+b}(p)$ for any $p\in S$, we know that $A$ contains at most $C(N, H, D, b, a)$ points which are $(b'-a')/3$ separated. 

In the following, we will alway assume $Y$ has one component.

By the uniqueness of solutions of the continuity equation \eqref{continuity-equ} (more precisely the local uniqueness, see \cite[Lemma 3.14]{CDNPSW}), we have that
for any $E\subset A$, $t\in [-(b-a)/3, (b-a)/3]$,

(i) for $m=0$, $K=0$,
\begin{equation}(F_t)_{\sharp}\mathfrak m(E)=\mathfrak m(E);\label{flow-vol1}\end{equation}

(ii) for $m\neq 0$ or $K\neq 0$, for a.e. $x\in A$, $\mu_t=(F_t)_{\sharp} \mathfrak m$ satisfies
\begin{equation}\frac{d}{dt} d\mu_t(x)+\Delta d_s(F_{-t}(x)) d\mu_t(x)=0\label{flow-vol2}\end{equation}
which implies that
\begin{equation}d\mu_t(x)=\frac{\op{sn}_H^{N-1}(d_s(x)-t+r_0)}{\op{sn}_H^{N-1}(d_s(x)+r_0)}d\mu_0(x).\label{flow-measure}\end{equation}

By the definition of $\mathfrak m_Y$, \eqref{measure-1}, \eqref{measure-2}, and the property \eqref{flow-vol1}, \eqref{flow-vol2} and the Laplacian estimates Theorem~\ref{lap-main},  as the proof of \cite[Proposition 5.28]{Gi13} (see also \cite[Lemma 5.11]{CDNPSW}), we have that for  $a'\leq c<d\leq b'$ and a Borel subset $\bar E\subset Y$, 

\begin{equation}\mathfrak m(E_c^d)=\begin{cases} \mathfrak m_Y(\bar E)(d-c), & m=0 \text{ and } K=0;\\
\mathfrak m_Y(\bar E)\int_{c}^{d} \op{sn}_H^{N-1}(t+r_0)dt, & m\neq 0 \text{ or } K\neq 0,\end{cases}\label{flow-measure-2}\end{equation}
where 
$$E_c^d=\{x\in A, \, c\leq d_s(x)\leq d, F_{-d_s(x)+a'}\in \bar E\}.$$

In fact, for $m\neq 0$ or $K\neq 0$, \eqref{flow-measure-2} can also be derived 
by \eqref{flow-measure} and the fact that
$$\int_{a'-d}^{a'-c}\mu_t(\bar E) dt=\mathfrak m(E_c^d).$$

Since $\Phi$ is a isometry, we know that $A$ has a warped product structure.  And thus $(a'+r_0, b'+r_0)\times_{\op{sn}_H} Y$ is a $\op{CD}_{loc}(K, N)$-space for $m\neq 0$ or $K\neq 0$, $(a', b')\times Y$ is a $\op{CD}_{\op{loc}}(0, N)$-space for $m=0$ and $K=0$.  

In the following, we denote
$$(Y_w, d_w, \mathfrak m_w)=\begin{cases} (\Bbb R\times Y, d, \mathcal L^1\otimes \mathfrak m_Y), & m=0, K=0;\\
(C(Y), d_K, \mathfrak m_N), & m\neq 0 \text{ or } K\neq 0.\end{cases} $$

%In the following ,we will show the product $(Y, d_Y, \mathfrak m_Y)$ is a $\RCD^*$-space.

%First we show that $(Y_w, d_w, \mathfrak m_w)$ is Infinitesimally Hilbertian.

\subsection{Properties of the metric measure space $(Y, d_Y, \mathfrak m_Y)$} 

In this subsection, we consider the metric measure space  $(Y, d_Y, \mathfrak m_Y)$ as above and as in \cite{CDNPSW} we will show that
\begin{Thm} \label{measure-pro}
Consider the metric measure space $(Y, d_Y, \mathfrak m_Y)$ defined as in above subsection. We have that $(Y, d_Y, \mathfrak m_Y)$ is infinitesimally Hilbertian, satisfies the almost everywhere locally doubling property, supports a local Poincar\'e inequality and is a measured-length space.
\end{Thm}

For any $[a'', b'']\subset [a', b']$, let 
$$T: [a'', b'']\times Y\to A, \, (t, y)\mapsto T(t, y)=F_{t}(y),$$ 
$$\hat T: C([0, 1], Y)\times [a'', b'']\to C([0, 1], X), \, (\bar \gamma_s, t)\mapsto \hat T(\bar\gamma, t)_s=F_{t}(\bar\gamma_s).$$
And let
$$P: A_{a, b}(S)\to Y, \,  x\mapsto P(x)=F_{-d_s(x)+a'}(x),$$
$$\hat P: C([0, 1], A_{a, b}(S))\to C([0, 1], Y), \, \gamma_s\mapsto \hat P(\gamma)_s=F_{-d_s(\gamma_s)+a'}(\gamma_s).$$

Consider the inclusion map $\iota: Y\to A_{a, b}(S)$. By the definition of $d_Y$, 
for $m=0, K=0$, $\iota$ is an isometric embedding; for $m\neq 0$ or $K\neq 0$,  $\iota$ is a Lipschitz map.

Since Poincar\'e inequality is invariant under a bi-Lipschitz map (cf. \cite[Section 4.3]{BB}, more precisely the proof of \cite[Proposition 4.16]{BB}) and $\iota: Y\to A_{a', b'}(S)$ is Lipschitz, the local Poincar\'e inequality in $A_{a', b'}(S)$ implies the local Poincar\'e inequality in $(Y, d_Y, \mathfrak m_Y)$. And since $A_{a', b'}(S)\subset X$ which is a $\RCD$-space, by \cite{Ra}, $A_{a', b'}(S)$ supports a local Poincar\'e inequality.

To see $(Y, d_Y, \mathfrak m_Y)$ is infinitesimally Hilbertian we only need to show that

\begin{Lem} \label{inf-hil}
Given a nonnegative function $h\in \op{Lip}(\Bbb R)$ with $h(t)=0$ for $t\in \Bbb R\setminus [a+\frac{b-a}4, b-\frac{b-a}4]$ and $h(t)=1$ for  $t\in [a', b']$, for each $g\in L^2(Y, \mathfrak m_Y)$, define $f(x)=g(P(x))h(d_s(x))$. Then
$g\in W^{1,2}(Y, d_Y, \mathfrak m_Y)$ if and only if $f\in W^{1,2}(X, d, \mathfrak m)$ and for $x\in A$, 

(i) for $m=0$, $K=0$,
$$|\nabla f|(x)= |\nabla g|(F_{-d_s(x)+a'}(x))=|\nabla^Y g|(F_{-d_s(x)+a'}(x));$$

(ii) for $m\neq 0$ or $K\neq 0$,
$$|\nabla f|(x)=\frac{\op{sn}_H(a'+r_0)}{\op{sn}_H(d_s(x)+r_0)}|\nabla g|(F_{-d_s(x)+a'}(x))=\frac{1}{\op{sn}_H(d_s(x)+r_0)}|\nabla^Y g|(F_{-d_s(x)+a'}(x)).$$
\end{Lem}
\begin{proof}
The proof is similar as in \cite[Proposition 5.12, Theorem 5.13]{CDNPSW}. 

Consider test plans $\bar \Pi$ on $Y$. Let $\Pi=\hat T_{\sharp} \left(\bar \Pi\times (b''-a'')^{-1}\mathcal L^1_{[a'', b'']}\right)$, where $[a'', b'']\subset [a', b']$ as above.

Claim 1: $\Pi$ is a test plan of $X$.

First note that
$$\int_{C([0,1], X)}\int_0^1|\dot{\gamma}_s|^2ds d\Pi(\gamma)=\int_{C([0, 1], Y)}(b''-a'')^{-1}\int_{a''}^{b''}\int_0^1|\hat T(\bar \gamma, t)'_s|^2ds dt d\bar\Pi(\bar \gamma),$$
and
\begin{equation*}
|\hat T(\bar \gamma, t)'_s|  = \lim_{h\to 0}\frac{d(\hat T(\bar \gamma, t)_{s+h}, \hat T(\bar \gamma, t)_s)}{|h|}
=  \lim_{h\to 0}\frac{d(F_t(\bar \gamma_{s+h}), F_t(\bar \gamma_s))}{|h|}.
\end{equation*}

By Theorem~\ref{flow-con}, 
for $m=0$, $K=0$, 
$$|\hat T(\bar \gamma, t)'_s| =|\dot {\bar\gamma}_s|,$$
and thus
$$\int_{C([0,1], X)}\int_0^1|\dot{\gamma}_s|^2ds d\Pi(\gamma)=\int_{C([0, 1], Y)}\int_0^1|\bar \gamma'_s|^2ds d\bar\Pi(\bar \gamma)<\infty;$$

For $m\neq 0, K=0$,
$$|\hat T(\bar \gamma, t)'_s| =\lim_{h\to 0}\frac{r_0+a'+t}{r_0+a'}\frac{d(\bar \gamma_{s+h}, \bar \gamma_s)}{|h|}=\frac{t+r_0+a'}{r_0+a'}|\dot{\bar \gamma}_s|$$
and thus
$$\int_{C([0,1], X)}\int_0^1|\dot{\gamma}_s|^2ds d\Pi(\gamma)=\frac{(r_0+a'+b')^2}{(r_0+a')^2} \int_{C([0, 1], Y)}\int_0^1|\bar\gamma'_s|^2ds d\bar\Pi(\bar \gamma)<\infty.$$

For $K=(N-1)$, 
$$|\hat T(\bar \gamma, t)'_s| =\lim_{h\to 0}\frac{|\sin(r_0+a'+t)|}{|\sin(r_0+a')|}\frac{d(\bar \gamma_{s+h}, \bar \gamma_s)}{|h|}=\frac{|\sin(r_0+a'+t)|}{|\sin(r_0+a')|} |\dot{\bar \gamma}_s|$$
and thus
$$\int_{C([0,1], X)}\int_0^1|\dot{\gamma}_s|^2ds d\Pi(\gamma)\leq \frac1{|\sin(r_0+a')|^2}\int_{C([0, 1], Y)}\int_0^1|\bar\gamma'_s|^2ds d\bar\Pi(\bar \gamma)<\infty.$$

For $K=-(N-1)$, it is similar as the case $K=N-1$.

And for the set $E_c^d=\{x\in A_{a, b}(S), \, a'\leq c\leq d_s(x)\leq d\leq b', x\in E\}$, where $E$ is a Borel set in $Y$, 
\begin{eqnarray*}
\left(e_t\right)_{\sharp}\Pi(E_c^d) &=& \bar \Pi\times (b''-a'')^{-1}\left.\mathcal L^1\right|_{[a'', b'']}((e_t\circ \hat T)^{-1} E_c^d)\\
&=& \bar \Pi(e_t^{-1}E)(b''-a'')^{-1}\left.\mathcal L^1\right|_{[a'', b'']}([c, d])\\
&\leq& C\mathfrak m_Y(E)\leq C' \mathfrak m(E_c^d).
\end{eqnarray*}

Claim 2: If $f\in W^{1,2}(X, d, \mathfrak m)$, then $g\in W^{1,2}(Y, d_Y, \mathfrak m_Y)$.

Assume $f\in W^{1,2}(X, d, \mathfrak m)$, then
\begin{eqnarray*}
\int_{C([0,1], Y)}g(\bar \gamma_1)-g(\bar\gamma_0)d\bar \Pi(\bar \gamma)&=& \int_{C([0,1], Y)}(b''-a'')^{-1}\int_{a''}^{b''}g(\bar \gamma_1)h(t)-g(\bar\gamma_0)h(t)dtd\bar\Pi(\bar \gamma)\\
& =& \int_{C([0, 1], X)} f(\gamma_1)-f(\gamma_0) d\Pi(\gamma)\\
&=& \int_{C([0,1], X)}\int_0^1 |\nabla f||\dot \gamma_s| ds d\Pi(\gamma)\\
&=& \int_{C([0,1], Y)}(b''-a'')^{-1}\int_{a''}^{b''} \int_0^1 |\nabla f|(\hat T(\bar \gamma_s, t))|\hat T(\bar \gamma, t)'_s| ds dt d\bar \Pi(\bar\gamma),
\end{eqnarray*}
and thus $g\in W^{1,2}(Y, d_Y, \mathfrak m_Y)$ and 
$$|\nabla g|(y)\leq (b''-a'')^{-1}\int_{a''}^{b''}|\nabla f|(T(y, t))\op{Lip}(F_t)dt.$$

Let $a''\to t_0,  b''\to t_0$. 

For $m=0, K=0$,
$$|\nabla g|(y)\leq |\nabla f|(F_{t_0}(y)).$$

For $m\neq 0$, $K=0$,
$$|\nabla g|(y)\leq \frac{t_0+r_0+a'}{r_0+a'} |\nabla f|(F_{t_0}(y)).$$

For $K\neq 0$,
$$|\nabla g|(y)\leq \frac{\op{sn}_H(t_0+r_0+a')}{\op{sn}_H(r_0+a')} |\nabla f|(F_{t_0}(y)).$$

Claim 3: If $g\in W^{1,2}(Y, d_Y, \mathfrak m_Y)$, then $f\in W^{1,2}(X, d, \mathfrak m)$.

By the definition of $f$, we only need to consider  text plan $\Pi$ which is supported in $C([0, 1], A_{a, b}(S))$.

In fact, for $\gamma\in C([0, 1], X)$ with finite length, without loss of generality we may assume $\gamma(0)$, $\gamma(1)\in A_{a+(b-a)/4, b-(b-a)/4}(S)$, we can divide $\gamma$ into finite pieces $\gamma^1, \gamma^2, \cdots, \gamma^{n+1}$ by taking cut points
$$0<t_1<t_2<\cdots <t_n<1,$$
such that for $A_1=\partial B_a(S)$, $A_2=\partial B_b(S)$, $A_3=\partial B_{a+(b-a)/4}(S)$, $A_4=\partial B_{b-(b-a)/4}(S)$,

(1) $\gamma(t_i)\in \cup_1^4 A_j$;

(2) if $\gamma(t_i)\in A_j$ then $\gamma(t_{i+1})\notin A_j$;

(3) $\left.\gamma\right|_{[t_i, t_{i+1}]}\subset A_{a, b}(S)$ or $\left.\gamma\right|_{[t_i, t_{i+1}]}\subset X\setminus A_{a+(b-a)/4, b-(b-a)/4}(S)$.

 Assume the pieces $\gamma^{i1}, \cdots, \gamma^{ik}\subset A_{a, b}(S)$. Note that $f(\gamma(t_i))=0$ for each $i$, then
 $$|f(\gamma_1)-f(\gamma_0)|\leq |f(\gamma^{i1}_1)-f(\gamma^{i1}_0)|+\cdots + |f(\gamma^{ik}_1)-f(\gamma^{ik}_0)|.$$
 If for any $\gamma\in C([0, 1], A_{a, b}(S))$, we have 
$$|f(\gamma_1)-f(\gamma_0)|\leq \int_0^1 G(\gamma_t)|\dot\gamma_t| dt,$$
then  we can take 
$$\tilde G(x)=\begin{cases} G(x), & x\in A_{a, b}(S),\\
0, & x\in X\setminus A_{a, b}(S),
\end{cases}$$
such that
$$|f(\gamma_1)-f(\gamma_0)|\leq \int_0^1\tilde G(\gamma_t)|\dot \gamma_t|dt.$$

Let $\bar \Pi=\hat P_{\sharp}(\Pi)$ where $\Pi$ is supported in $C([0, 1], A_{a, b}(S))$. 
As claim 1, we have that $\bar \Pi$ is a test plan of $Y$. To see this as above
$$\int_{C([0, 1], Y)}\int_0^1|\bar \gamma'_t|^2 dt d\bar \Pi=\int_{C([0, 1], X)}\int_0^1 \left|\frac{d}{dt}F_{-d_s(\gamma_t)+a'}(\gamma_t)\right|^2dtd\Pi.$$
For $m=0$, $K=0$, 
$$d^2(F_{-d_s(\gamma_{t+h})+a'}(\gamma_{t+h}), F_{-d_s(\gamma_{t})+a'}(\gamma_{t}))=d^2(\gamma_{t+h}, \gamma_t)-|d_s(\gamma_{t+h})-d_s(\gamma_t)|^2\leq d^2(\gamma_{t+h}, \gamma_t),$$
Thus
$$\int_{C([0, 1], Y)}\int_0^1|\bar \gamma'_t|^2 dt d\bar \Pi\leq \int_{C([0, 1], X)}\int_0^1 \left|\dot\gamma_t\right|^2dtd\Pi<\infty.$$

For $m\neq 0$, $K=0$, by
$$\frac{d^2(\gamma_{t+h}, \gamma_t)-(d_s(\gamma_{t+h})+r_0)^2-(d_s(\gamma_{t})+r_0)^2}{2(d_s(\gamma_{t+h})+r_0)(d_s(\gamma_{t})+r_0)}=\frac{d^2(F_{-d_s(\gamma_{t+h})+a'}(\gamma_{t+h}), F_{-d_s(\gamma_{t})+a'}(\gamma_{t}))-2(a'+r_0)^2}{2(a'+r_0)^2},$$
we have
$$d^2(F_{-d_s(\gamma_{t+h})+a'}(\gamma_{t+h}), F_{-d_s(\gamma_{t})+a'}(\gamma_{t}))\leq \frac{(a'+r_0)^2}{(d_s(\gamma_{t+h})+r_0)(d_s(\gamma_{t})+r_0)}d^2(\gamma_{t+h}, \gamma_t).$$
Thus
$$\int_{C([0, 1], Y)}\int_0^1|\bar \gamma'_t|^2 dt d\bar \Pi\leq \int_{C([0, 1], X)}\int_0^1 \left(\frac{a'+r_0}{d_s(\gamma_t)+r_0}\right)^2\left|\dot\gamma_t\right|^2dtd\Pi<\infty.$$

For $K=(N-1)$, by
$$\frac{\cos d(\gamma_{t+h}, \gamma_t)-\cos(d_s(\gamma_{t+h})+r_0)\cos(d_s(\gamma_{t})+r_0)}{\sin(d_s(\gamma_{t+h})+r_0)\sin(d_s(\gamma_{t})+r_0)}=\frac{\cos d(F_{-d_s(\gamma_{t+h})+a'}(\gamma_{t+h}), F_{-d_s(\gamma_{t})+a'}(\gamma_{t})) -\cos^2(a'+r_0)}{\sin^2(a'+r_0)}$$
i.e.
$$\frac{\cos d(\gamma_{t+h}, \gamma_t)-\cos(d_s(\gamma_{t+h})-d_s(\gamma_t))}{\sin(d_s(\gamma_{t+h})+r_0)\sin(d_s(\gamma_{t})+r_0)}=\frac{\cos d(F_{-d_s(\gamma_{t+h})+a'}(\gamma_{t+h}), F_{-d_s(\gamma_{t})+a'}(\gamma_{t})) -1}{\sin^2(a'+r_0)},$$
$$\cos d(\gamma_{t+h}, \gamma_t)-1\leq \frac{\sin(d_s(\gamma_{t+h})+r_0)\sin(d_s(\gamma_{t})+r_0)}{\sin^2(a'+r_0)}\left(\cos d(F_{-d_s(\gamma_{t+h})+a'}(\gamma_{t+h}), F_{-d_s(\gamma_{t})+a'}(\gamma_{t})) -1\right).$$
thus
$$\int_{C([0, 1], Y)}\int_0^1|\bar \gamma'_t|^2 dt d\bar \Pi\leq \int_{C([0, 1], X)}\int_0^1 \left(\frac{\sin(a'+r_0)}{\sin (d_s(\gamma_t)+r_0)}\right)^2\left|\dot\gamma_t\right|^2dtd\Pi<\infty.$$

For $K=-(N-1)$, similarly as the $K=(N-1)$ case we have
$$\int_{C([0, 1], Y)}\int_0^1|\bar \gamma'_t|^2 dt d\bar \Pi\leq \int_{C([0, 1], X)}\int_0^1 \left(\frac{\sinh(a'+r_0)}{\sinh (d_s(\gamma_t)+r_0)}\right)^2\left|\dot\gamma_t\right|^2dtd\Pi<\infty.$$

And for $E\subset Y$,
\begin{eqnarray*}
(e_t)_{\sharp} \bar \Pi(E)  &=& \Pi ((e_t\circ \hat P)^{-1}(E))\\
&=& \Pi (e_t^{-1}(E_a^b))=(e_t)_{\sharp}\Pi(E_a^b)\\
&\leq & c\mathfrak m(E_a^b)\leq c' \mathfrak m_Y(E).
\end{eqnarray*}

Now, since $g\in W^{1,2}(Y, d_Y, \mathfrak m_Y)$ and 
\begin{eqnarray*}
& & \int_{C([0,1], X)} f(\gamma_1)-f(\gamma_0) d\Pi(\gamma)\\
& =& \int_{C([0,1], X)}g(F_{-d_s(\gamma_1)+a'}(\gamma_1))h(d_s(\gamma_1))-g(F_{-d_s(\gamma_0)+a'}(\gamma_0))h(d_s(\gamma_0)) d\Pi(\gamma)\\
&\leq & \int_{C([0,1], Y)}g(\bar \gamma_1)-g(\bar \gamma_0) d\bar \Pi(\bar \gamma)+ \int_{C([0,1], X)}g(F_{-d_s(\gamma_0)+a'}(\gamma_0))\int_0^1h'(d_s(\gamma_t))|\dot \gamma_t| dtd\Pi(\gamma)\\
&\leq & \int_{C([0,1], Y)}\int_0^1|\nabla g|(\bar \gamma_t) |\bar \gamma'_t|dtd\bar \Pi(\bar \gamma)+ \int_{C([0,1], X)}g(F_{-d_s(\gamma_0)+a'}(\gamma_0))\int_0^1h'(d_s(\gamma_t))|\dot \gamma_t| dtd\Pi(\gamma)\\
&\leq & \int_{C([0,1], X)}\int_0^1|\nabla g|(F_{-d_s(\gamma_t)+a'}(\gamma_t))|\hat P(\gamma)'_t|+g(F_{-d_s(\gamma_0)+a'}(\gamma_0))h'(d_s(\gamma_t))|\dot \gamma_t| dtd\Pi(\gamma)
\end{eqnarray*}

Thus for $x\in A$,
$$|\nabla f|(x)\leq |\nabla g|(F_{-d_s(x)+a'}(x))\op{Lip}(\hat P).$$
For $m=0$, $K=0$, 
$$|\nabla f|(x)\leq |\nabla g|(F_{-d_s(x)+a'}(x));$$
For $m\neq 0$, $K=0$, 
$$|\nabla f|(x)\leq \frac{a'+r_0}{d_s(x)+r_0}|\nabla g|(F_{-d_s(x)+a'}(x));$$
For $K\neq 0$,
$$|\nabla f|(x)\leq \frac{\op{sn}_H(a'+r_0)}{\op{sn}_H(d_s(x)+r_0)}|\nabla g|(F_{-d_s(x)+a'}(x)).$$
\end{proof}

Recall the definitions of almost everywhere locally doubling property and measured-length property.
\begin{Def}
A metric measure space $(X, d, \mathfrak m)$ is almost everywhere locally doubling if there is a full measure Borel subset $\hat X\subset X$ satisfying that for each $x\in \hat X$, there exists an open set $U\ni x$, constants $C, R>0$, such that for $r\in (0, R)$, $y\in U$,
$$\mathfrak m(B_{2r}(y))\leq C\mathfrak m(B_{r}(y)).$$
\end{Def}
\begin{Def}
A metric measure space $(X, d, \mathfrak m)$ is measured-length if there is a full measure subset $\hat X\subset X$ satisfying the following: For $x_0, x_1\in \hat X$, there exist $\epsilon>0$, a map 
$$(0, \epsilon]^2\to \mathcal P(C([0, 1], X)), \, (t_0, t_1)\mapsto \Pi^{t_0, t_1},$$
such that

(i) For $\phi\in C_{b}(C([0, 1], X))$, the map 
$$(0, \epsilon]^2\to \Bbb R, \, (t_0, t_1)\mapsto \int \phi d\Pi^{t_0, t_1}$$
is Borel;

(ii) For $i=0, 1$, 
$$(e_i)_{\sharp}\Pi^{t_0, t_1}=\frac{1_{B_{t_i}(x_i)}}{\mathfrak m(B_{t_i}(x_i))}\mathfrak m;$$

(iii) 
$$\limsup_{t_0, t_1\downarrow 0}\int \int_0^1 |\dot \gamma_t|^2 dt d\Pi^{t_0, t_1}(\gamma)\leq d^2(x_0, x_1).$$
\end{Def}

Now we prove $(Y, d_Y, \mathfrak m_Y)$ is almost everywhere locally doubling and measured-length.
\begin{Lem}
The metric measure space $(Y, d_Y, \mathfrak m_Y)$ is almost everywhere locally doubling.
\end{Lem}
\begin{proof}%[Proof of Theorem~\ref{measure-pro}]

For $E\subset Y$, let $E_r=\{x\in A, \, 0\leq d_s(x)-a'\leq r, F_{-d_s(x)+a'}\in E\}$.

For each $x\in Y$, there is open set $Y\supset U\ni  x$, $R>0$, such that for each $ y\in U$,  $r<R$, 

(i) for $m=0$, $K=0$,
\begin{eqnarray*}
\mathfrak m_Y(B_{2r}( y))& =&\frac{\mathfrak m(B_{2r}(y)_{2r})}{2r}\leq \frac{\mathfrak m(B_{4r}(F_r(\iota(y))))}{2r}\\
&\leq& \frac{\svolball{H}{4r}}{\svolball{H}{r}}\frac{\mathfrak m(B_{r}(F_r(\iota(y))))}{2r}\leq \frac{\svolball{H}{4r}}{\svolball{H}{r}}\frac{\mathfrak m(B_{r}(\iota(y))_{2r})}{2r}\\
& \leq & \frac{\svolball{H}{4r}}{\svolball{H}{r}}\mathfrak m_Y(B_{r}( y).
\end{eqnarray*}
(ii) for $m\neq 0$ or $K\neq 0$, note that for any $r<R$, there are $0<c(K, R)< C(K, R)$ such that 
$$B_{c(K, R)r}(F_r(\iota(y)))\subset B_r(y)_{2r}\subset B_{2r}(y)_{2r}\subset B_{C(K, R)r}(F_r(\iota(y))).$$
Then by relative volume comparison, a similar argument as the $m=0$, $K=0$ case gives the almost everywhere locally doubling property.
\end{proof}

\begin{Lem}
The metric measure space $(Y, d_Y, \mathfrak m_Y)$ is a measured-length space.
\end{Lem}

The proof of this lemma is the same as the one \cite[Proposition 5.14]{CDNPSW}. Here we omit it. And by the following theorem, we can derive a series of properties about $(Y_w, d_w, \mathfrak m_w)$.

\begin{Thm}[\cite{GH, CDNPSW}] \label{glob-hil} 
Consider a warped product space $Y_w=I\times_w Y$ where $I$ is a bounded interval in $\Bbb R$ and $w_d, w_m: I\to [0, \infty)$ with $w_m>0$ for points in the interior of $I$. Assume $(Y, d_Y, \mathfrak m_Y)$ is a.e. locally doubling, measured length, infinitesimally Hilbertian,  then $(Y_w, d_w, \mathfrak m_w)$ is a.e. locally doubling, measured length, infinitesimally Hilbertian and it has the Sobolev to Lipschitz property.
\end{Thm}

\subsection{$(Y, d_Y, \mathfrak m_Y)$ is a $\RCD$-space}
%In this subsection, we will show that $(Y, d_Y, \mathfrak m_Y)$ is a $\RCD$-space. 

For $m=0$, $K=0$, we will show that $(\Bbb R\times Y, d, \mathcal L^1\otimes \mathfrak m_Y)$ satisfies the $\op{CD}_{\op{loc}}(0, N)$ condition. Then by the local-to-global property \cite[Theorem 3.14]{EKS} and \cite{RS}, we know that $\Bbb R\times Y$ is essentially non-branching and is a $\op{CD}(0, N)$-space.  Finally an argument as in \cite{Gi13} gives that $(Y, d_Y, \mathfrak m_Y)\in \RCD(0, N-1)$.

For $K\neq 0$ or $m\neq 0$, Ketterer \cite[Theorem 1.2]{Ket1} (see Theorem~\ref{cone-rcd}) proved that if $(C(Y), d_K, \mathfrak  m_N)\in \RCD^*(K, N)$, then $(Y, d_Y, \mathfrak m_Y)\in \RCD^*(N-2, N-1)$ and $\op{diam}(Y)\leq \pi$. We have known that:
$(a'+r_0, b'+r_0)\times_{\op{sn}_H(r)} Y$ is isometric to $A_{a', b'}(S)\subset X$ which is a $\op{RCD}(K, N)$-space. 
We will show that Ketterer's result (ii) of Theorem~\ref{cone-rcd} holds under this weaker condition. 

First recall some basic definitions about Dirichlet forms. See \cite{AGS} or \cite{Ket1} for more details. 

Let $(X, d, \mathfrak m)$ be a locally compact, separable Hausdorff metric measure space. A symmetric Dirichlet form $\mathcal E^X$ defined in $D(\mathcal E^X)\subset L^2(X, \mathfrak m)$ is a $L^2(X, \mathfrak m)$-lower semi-continuous, quadratic form that satisfies the Markov property. The domain $D(\mathcal E^X)$ is a Hilbert space with respect to 
$$(u, u)_{D(\mathcal E^X)}=(u, u)_{L^2(X, \mathfrak m)}+\mathcal E^X(u, u).$$
There is a self-adjoint, negative-definite operator $(L^X, D_2(L^X))$ on $L^2(X, \mathfrak m_X)$ where 
$$D_2(L^X)=\{u\in D(\mathcal E^X),\, \exists v\in L^2(X, \mathfrak m), -(v, w)_{L^2(X, \mathfrak m)}=\mathcal E^X(u, w), \forall w\in D(\mathcal E^X)\}.$$
Let $v=L^X u$. 

Denote $D^{\infty}(\mathcal E^X)=D(\mathcal E^X)\cap L^{\infty}(X, \mathfrak m)$. For $u, \phi\in D^{\infty}(\mathcal E^X)$, define
$$\Gamma^X(u; \phi)=\mathcal E^X(u, u\phi)-\frac12\mathcal E^X(u^2, \phi),$$
which can be extended by continuity to $u\in D(\mathcal E^X)$.

For $u, v\in D(\mathcal E^X), \phi\in D^{\infty}(\mathcal E^X)$, define
$$\Gamma^X(u, v; \phi)=\frac12\left(\Gamma^X(u;\phi)+\Gamma^X(v;\phi)-\Gamma^X(u-v; \phi)\right).$$
And let
$$2\Gamma_2^X(u, v;\phi)=\Gamma^X(u, v; L^X\phi)-2\Gamma^X(u, L^Xv; \phi), \quad \Gamma^X_2(u;\phi)=\Gamma^X_2(u, u;\phi).$$
where
$u, v\in D(\Gamma^X_2)=\{u\in D_2(L^X), \, L^Xu\in D(\mathcal E^X)\}$, test function $\phi\in D_+^{b, 2}(L^X)=\{\phi\in D_2(L^X),\, \phi, L^X\phi\in L^{\infty}(X, \mathfrak m), \phi>0\}$.

Let $D'$ be the set of $u$ such that the map $\phi\mapsto \Gamma^X(u;\phi)$ is an absolutely continuous measure w.r.t. $\mathfrak m$ which is denoted by $\Gamma^X(u)\mathfrak m$. If $D'=D(\mathcal E^X)$, we call $\mathcal E^X$ admits a ``carr\'e du champ" operator.

If $\mathcal E^X$ is strongly local and admits a ``carr\'e du champ" operator (see \cite[Section 2.1]{Ket1}), one can define $D_{\op{loc}}(\mathcal E^X)\subset L^2_{\op{loc}}(X, \mathfrak m)$ and thus there is a intrinsic distance of $\mathcal E^X$,
$$d_{\mathcal E^X}(x, y)=\sup\{u(x)-u(y), \, u\in D_{\op{loc}}(\mathcal E^X)\cap C(X), \Gamma^X(u)\leq 1, \mathfrak m-a.e.\}.$$

Note that if $(X, d, \mathfrak m)$ is infinitesimally Hilbertian, the Cheeger energy of $X$, $\op{Ch}^X$ is a symmetric Dirichlet form. And the corresponding $L^Xu=\Delta_Xu$, the Laplacian of $u$, $\Gamma^X$ is the same as the one defined in Subsection 2.1. Compared with $\Gamma_2$ defined in $\RCD(K, N)$ (Subsection 2.5), $\Gamma_2^X$ here is in a weak sense.

\begin{Thm} \label{Y-rcd}
Let $(Y, d_Y, \mathfrak m_Y)$ be as the one in the beginning of this section.  

(i) For $m=0$, $K=0$, $(\Bbb R\times Y, d, \mathcal L^1\otimes \mathfrak m_Y)$ is a $\RCD(0, N)$-space and thus $(Y, d_Y, \mathfrak m_Y)$ is a $\RCD(0, N-1)$-space;

(ii) For $m\neq 0$, $K=0$,  $(C(Y), d_0, \mathfrak m_N)$ is a $\RCD(0, N)$-space and $(Y, d_Y, \mathfrak m_Y)$ is a $\RCD(N-2, N-1)$-space;

(iii) For $K\neq 0$, $(Y, d_Y, \mathfrak m_Y)$ is a $\RCD(N-2, N-1)$-space. Especially for $K>0$, $(C(Y), d_K, \mathfrak m_N)$ is a $\RCD(K, N)$-space.
\end{Thm}

\begin{proof}
%To prove (i), by Theorem~\ref{measure-pro}, Theorem~\ref{glob-hil} and local to global property, we only need to show that $(\Bbb R\times Y, d, \mathcal L^1\otimes \mathfrak m_Y)$ is a $\op{CD}_{\op{loc}}(0, N)$-space.

For $m=0$ and $K=0$, note that for $\delta>0$, $(a'+\delta, b'-\delta)\times Y$ is locally a $\op{CD}(0, N)$-space. And for any $t\in \Bbb R$, $(t, t+b'-a'-2\delta)\times Y$ is isometric to $(a'+\delta, b'-\delta)\times Y$ and thus is locally a $\op{CD}(0, N)$-space. Since $\Bbb R\times Y$ can be covered by $(t, t+b'-a'-2\delta)\times Y$, $t\in \Bbb R$, we know $(\Bbb R\times Y, d, \mathcal L^1\otimes \mathfrak m_Y)$ is in $\op{CD}_{\op{loc}}(0, N)$.

By the discussion above Theorem 3.14 in \cite{EKS} and the infinitesimaly Hilbertian property of $\Bbb R\times Y$ (Theorem~\ref{measure-pro} and Theorem~\ref{glob-hil}), we know that $\Bbb R\times Y$ is essentially non-branching (see \cite{RS}). And then by local to global property \cite{CMi}, $(\Bbb R\times Y, d, \mathcal L^1\otimes \mathfrak m_Y)$ is $\RCD(0, N)$-space. 

Now
$(Y, d_Y, \mathfrak m_Y)$ is a $\RCD(0, N-1)$-space by the argument in \cite{Gi13}.

For $m\neq 0$ or $K\neq 0$, first endow $Y$ with a new metric $d'_Y$ such that $\op{diam}(Y)\leq \pi$:
$$d'_Y(y_1, y_2)=\begin{cases} d_Y(y_1, y_2), & \text{ if } d_Y(y_1, y_2)\leq \pi;\\
 \pi, & \text{ if } d_Y(y_1, y_2)>\pi.\end{cases}
$$
Then $(Y, d_Y)$ is locally isometric to $(Y, d'_Y)$ and the $(K, N)$-cone $C(Y, d_Y)=C(Y, d'_Y)$. In the following we will assume $Y$ endowed with the metric $d'_Y$.

For $m\neq 0$, $K=0$, note that for any $r>0$, $(C(Y), r^{-1} d_0, c_r \mathfrak m_N)$ is isometric to $(C(Y),  d_0,  \mathfrak m_N)$, where 
$$c_r=\left(\int_{B_r(x)} 1-\frac{d_0(x, y)}{r} d\mathfrak m_N(y)\right)^{-1}.$$
Then for any $t>0$, $(a'+r_0+\delta, b'+r_0-\delta)\times_r Y\subset (C(Y), (a'+r_0+\delta)/t d_0, c_{t/(a'+r_0+\delta)}\mathfrak m_N)$ is isometric to the one in $(C(Y),  d_0,  \mathfrak m_N)$ which is locally a $\op{CD}(0, N)$-space. Rescaling back, we can see that
$$\left(t, \frac{b'+r_0-\delta}{a+r_0+\delta}t\right)\times_r Y\in (C(Y),  d_0,  \mathfrak m_N)$$
is in $\op{CD}_{\op{loc}}(0, N)$ which implies $(C(Y),  d_0,  \mathfrak m_N)\in \op{CD}_{\op{loc}}(0, N)$. As the above discussion we know that $(C(Y), d_0, \mathfrak m_N)$ is a $\RCD(0, N)$-space. Then (ii) of Theorem~\ref{cone-rcd} implies that $(Y, d'_Y, \mathfrak m_Y)$ is a $\RCD(0, N-1)$-space. And so is $(Y, d_Y, \mathfrak m_Y)$.

For (iii), we follow the argument in the proof of \cite[Theorem 1.2]{Ket1}. 

By Theorem~\ref{measure-pro}, we know that $(Y, d_Y, \mathfrak m_Y)$ is infinitesimally Hilbertian and satisfies the Sobolev to Lipschitze property. Thus by \cite{EKS} to prove $(Y, d_Y, \mathfrak m_Y)$ is a $\RCD(N-2, N-1)$-space,  one only need to show it satisfying the $(N-2, N-1)$ Bakry-Ledoux estimate:
for any $f\in W^{1,2}(Y, d_Y, \mathfrak m_Y)$, $t>0$,
\begin{equation}|\nabla (H_t(f))|^2+\frac{4(N-2)t^2}{(N-1)(e^{2(N-2)t}-1)}|\Delta H_t(f)|^2\leq e^{-2(N-2)t}H_t(|\nabla f|^2), \mathfrak m_Y-a.e.  \label{B-L}\end{equation}

As in the proof of \cite[Theorem 1.2]{Ket1},  we can derive \eqref{B-L} by the following kind of Bakry-Emery inequality: for any $u\in D(\Gamma^Y_2)$, \begin{eqnarray}
 & & \frac12\int_Y L^Y\phi \Gamma^Y(u) d\mathfrak m_Y-\int_Y \Gamma^Y(u, L^Y u)\phi d\mathfrak m_Y \nonumber\\
& \geq & (N-1)\int_Y \Gamma^Y(u)\phi d\mathfrak m_Y+\frac1{N}\int_Y (L^Y u)^2 \phi d\mathfrak m_Y -\frac1{(N+1)N}\int_F(L^Y u+N u)^2\phi d\mathfrak m_Y. \label{B-E}
 \end{eqnarray}
 Namely the methods and calculations from \eqref{B-E} to \eqref{B-L}  are  similar as the proof of the equivalence of Bakry-Emery inequality and Bakry-Ledoux estimate in \cite{EKS} (see \cite[Proposition 4.7, 4.9]{EKS}, see also the proof of \cite[Theorem 1.2]{Ket1}). Here we omit it and only point out that in the proof one needs the regularity of $\op{Ch}^Y$: The intrinsic distance of $\op{Ch}^Y$, $d_{\op{Ch}^Y}=d_Y$ and $\op{Ch}^Y$ is strongly regular \cite[Lemma 5.14]{Ket1}.
 
In \cite[Lemma 5.14]{Ket1}, Ketterer assumed that $(C(Y), d_K, \mathfrak m_N)$ is a $\RCD^*$-space.  By examining the proof there, we can see that it only needs the Sobolev to Lipschitz property of  $(C(Y), d_K, \mathfrak m_N)$ which can be seen from Theorem~\ref{glob-hil}.
 
In the following, we will see how to derive \eqref{B-E} from the relations between the Cheeger energies $\op{Ch}^Y$, $\op{Ch}^{C(Y)}$ and the warped product $\mathcal E^C$, where $\mathcal E^C$ is a symmetric form on $L^2(C(Y), \op{sn}_H^{N-1}(t)dt\otimes \mathfrak m_Y)$ defined as
$$\mathcal E^C(u)=\int_Y \int_{I_K} |u'_y(t)|^2dt d\mathfrak m_Y+\int_{I_K} \op{Ch}^Y(u_p)\op{sn}_H^{N-3}(t) dt$$  
where $u\in C_0^{\infty}(I_K)\otimes D(\op{Ch}^Y)$ and $u_p=u(p, \cdot)$, $u_y=u(\cdot, y)$.

First note that:

Claim 1:  As in \cite[Lemma 5.11]{Ket1}, we have that the intrinsic distance $d_{\mathcal E^C}$ of $\mathcal E^C$ satisfies that, for $x, y\in A_{a', b'}(S)$,
$$d_{\mathcal E^C}(x, y)=d_K(x, y).$$

 Claim 2: As \cite[Corollary 5.12]{Ket1}, for $u\in D_{\op{loc}}(\mathcal E^C)\cap L^{\infty}(A_{a, b}(S))$, we have that 
 $$\mathcal E^C(u)=\op{Ch}^{C(Y)}(u).$$
 
 By the definition of $\mathcal E^C$, for $u_1\otimes u\in C_0^{\infty}((a', b'))\otimes D(\Gamma^Y_2), 1\otimes \phi\in 1\otimes D_{+}^{b, 2}(L^Y)$, a careful calculation as in \cite{Ket1} gives that (see also (33) in \cite{Ket1}):
 \begin{eqnarray}
& & \Gamma_2^C(u_1\otimes u; 1\otimes \phi) = \int_{C(Y)} \Gamma_2^{I_K, \op{sn}_H^{N-1}}(u_1)u^2\phi+\frac{u_1^2}{\op{sn}_H^2}\Gamma_2^Y(u)\phi+\frac12\Gamma^{I_K}(u_1)\frac{1}{\op{sn}_H^2}L^Y(u)\phi \nonumber\\
 & & +\int_{C(Y)}\left(\frac12L^{I_K, \op{sn}_H^{N-1}}\left(\frac{u_1^2}{\op{sn}_H^2}\right)\Gamma^Y(u)-\frac{u_1}{\op{sn}_H^2}L^{I_K, \op{sn}_H^{N-1}}(u_1)\Gamma^Y(u)-\Gamma^{I_K}\left(u_1, \frac{u_1}{\op{sn}_H^2}\right)uL^Y(u)\right)\phi, \label{BE-weak}
 \end{eqnarray}   
where
$L^{I_K}=d^2/d t^2$ and 
$$\mathcal E^{I_K, \op{sn}_H^{N-1}}(u)=\int_{I_K}(u')^2\op{sn}_H^{N-1} dt,$$ 
thus
$$L^{I_K, \op{sn}_H^{N-1}}(u)=u''+\frac{N-1}{\op{sn}_H}\op{sn}'_H u'.$$
By Claim 2, $\mathcal E^C=\op{Ch}^{C(Y)}$. And the Bakry-\'Emery inequality holds in $A_{a', b'}(S)$:
$$\Gamma_2^C(u_1\otimes u; 1\otimes \phi) \geq K\Gamma^C(u_1\otimes u; \phi)+\frac1{N}\int (L^C u_1\otimes u)^2\phi.$$
Then as in \cite[Theorem 3.10]{Ket1}, by taking $u_1=\sin t$, for $K=N-1$, and $u_1=\sinh t$, for $K=-(N-1)$,  $t\in (a', b')$ in \eqref{BE-weak} and using the above Bakry-\'Emery inequality,  we have \eqref{B-E}.
 
%Compared with \cite[Lemma 5.11]{Ket1} with  Claim 1 and \cite[Corollary 5.12]{Ket1} with Claim 2, Ketterer assumed  that $(C(Y), d_K, \mathfrak m_N)$ is $\RCD(K, N)$-space. Here we assume $A_{a', b'}(S)$ has curvature-dimension condition. 

To see Claim 1, note that in the proof of  \cite[Lemma 5.11]{Ket1}, the only place where Ketterer used the $\RCD^*$-condition of $C(Y)$ is to derive the inequality (see (47) in \cite{Ket1})
\begin{equation}
d_{\mathcal E^C}\leq d_K. \label{dist-comp}
\end{equation}
Here we want to show \eqref{dist-comp} holds locally in $A_{a', b'}(S)$. This can be seen by restricting all Ketterer's estimates in $A_{a', b'}(S)$ and then the curvature-dimension condition can be applied similarly.

In the proof of \cite[Corollary 5.12]{Ket2},  one only needs the local doubling property and local Poincar\'e inequality of $Y$ and $C(Y)$, and Claim 1. Thus Claim 2 is derived. 
\end{proof}

\end{document}